\documentclass[11pt,reqno]{amsart}

\newtheorem{thm}{Theorem}[section]
\newtheorem*{thm*}{Theorem}

\newtheorem*{lem*}{Lemma}

\newtheorem{cor}[thm]{Corollary}

\newtheorem{prop}[thm]{Proposition}

\theoremstyle{definition}

\newtheorem{assump}[thm]{Assumption}
\newtheorem*{case*}{Case}

\newtheorem{defn}[thm]{Definition}
\newtheorem*{defn*}{Definition}
\newtheorem{exmp}[thm]{Example}
\newtheorem*{exmp*}{Example}

\renewcommand{\thestep}{}

\theoremstyle{remark}

\newtheorem{case}{Case}\renewcommand{\thecase}{}

\newtheorem{rmk}[thm]{Remark}
\newtheorem*{rmk*}{Remark}

\makeatletter
\def\alphenumi{
  \def\theenumi{\alph{enumi}}
  \def\p@enumi{\theenumi}
  \def\labelenumi{(\@alph\c@enumi)}}
\makeatother

\makeatletter
\def\thecase{\@arabic\c@case}
\makeatother

\makeatletter
\def\thestep{\@arabic\c@step}
\makeatother

\newcount\hh
\newcount\mm
\mm=\time
\hh=\time
\divide\hh by 60
\divide\mm by 60
\multiply\mm by 60
\mm=-\mm
\advance\mm by \time
\def\hhmm{\number\hh:\ifnum\mm<10{}0\fi\number\mm}

\let\oldmarginpar\marginpar
\renewcommand\marginpar[1]{\-\oldmarginpar[\raggedleft\footnotesize #1]%
{\raggedright\footnotesize #1}}

\newcommand\EE{\mathbb{E}}

\newcommand\HH{\mathbb{H}}

\newcommand\PP{\mathbb{P}}
\newcommand\QQ{\mathbb{Q}}
\newcommand\RR{\mathbb{R}}
\newcommand\SSS{\mathbb{S}}

\newcommand\sA{{\mathscr{A}}}
\newcommand\sB{{\mathscr{B}}}
\newcommand\sC{{\mathscr{C}}}

\newcommand\sF{{\mathscr{F}}}
\newcommand\sG{{\mathscr{G}}}

\newcommand\eps{\varepsilon}

\newcommand\less{\setminus}

\newcommand\loc{\operatorname{loc}}

\newcommand\supp{\operatorname{supp}}

\numberwithin{equation}{section}

\usepackage{amssymb}
\usepackage{hyperref}
\usepackage{mathrsfs}
\usepackage{verbatim}
\usepackage[usenames]{color}
\usepackage{url}

\hypersetup{pdftitle={On the martingale problem for degenerate-parabolic partial differential operators with unbounded coefficients and a mimicking theorem for It\^o processes}}
\hypersetup{pdfauthor={Paul M. N. Feehan and Camelia Pop}}

\begin{document}

\title[Martingale problem for degenerate-parabolic partial differential operators]{On the martingale problem for degenerate-parabolic partial differential operators with unbounded coefficients and a mimicking theorem for It\^o processes}

\author[P. M. N. Feehan]{Paul M. N. Feehan}
\address[PF]{Department of Mathematics, Rutgers, The State University of New Jersey, 110 Frelinghuysen Road, Piscataway, NJ 08854-8019}
\email[PF]{feehan@math.rutgers.edu}

\author[C. Pop]{Camelia Pop}
\address[CP]{Department of Mathematics, University of Pennsylvania, 209 South 33rd Street, Philadelphia, PA 19104-6395}
\email{cpop@math.upenn.edu}

\date{August 7, 2013}

\begin{abstract}
Using results from our companion article \cite{Feehan_Pop_mimickingdegen_pde} on a Schauder approach to existence of solutions to a degenerate-parabolic partial differential equation, we solve three intertwined problems, motivated by probability theory and mathematical finance, concerning degenerate diffusion processes. We show that the martingale problem associated with a degenerate-elliptic differential operator with unbounded, locally H\"older continuous coefficients on a half-space is well-posed in the sense of Stroock and Varadhan. Second, we prove existence, uniqueness, and the strong Markov property for weak solutions to a stochastic differential equation with degenerate diffusion and unbounded coefficients with suitable H\"older continuity properties. Third, for an It\^o process with degenerate diffusion and unbounded but appropriately regular coefficients, we prove existence of a strong Markov process, unique in the sense of probability law, whose one-dimensional marginal probability distributions
match those of the given It\^o process.
\end{abstract}

%
%
%
%



\thanks{PF was partially supported by NSF grant DMS-1059206. CP was partially supported by a Rutgers University fellowship. }

\maketitle
\tableofcontents

\section{Introduction}
\label{sec:Intro}
Consider a time-dependent, degenerate-elliptic differential operator defined by \emph{unbounded} coefficients $(a,b)$ on the half-space $\HH := \RR^{d-1}\times(0,\infty)$ with $d\geq 1$,
\begin{equation}
\label{eq:MartingaleGenerator}
\sA_tv(x) := \frac{1}{2}\sum_{i,j=1}^d x_da_{ij}(t,x)v_{x_ix_j}(x) + \sum_{i=1}^d b_i(t,x)v_{x_i}(x), \quad (t,x) \in [0,\infty)\times\HH,
\end{equation}
and $a=(a_{ij})$, $b=(b_i)$, and $v \in C^{2}(\overline\HH)$. The operator $\sA_t$ becomes \emph{degenerate} along the boundary $\partial\HH = \{x_d=0\}$ of the half-space. In this article, motivated by applications to probability theory and mathematical finance \cite{Antonov_Misirpashaev_Piterbarg_2009, Dupire1994, Gyongy, Piterbarg_markovprojection}, we apply the main result of our companion article\footnote{Our longer previous manuscript \cite{Feehan_Pop_mimickingdegen} combined \cite{Feehan_Pop_mimickingdegen_pde} with the present article.} \cite{Feehan_Pop_mimickingdegen_pde} to solve three intertwined problems concerning degenerate diffusion processes related to \eqref{eq:MartingaleGenerator}.

We show that the martingale problem \S \ref{subsubsec:MainMartingleProblem} for the degenerate-elliptic operator with unbounded coefficients, $\sA_t$, in \eqref{eq:MartingaleGenerator} is well-posed in the sense of D. W. Stroock and S. R. S. Varadhan \cite{Stroock_Varadhan}. In addition, as discussed in more detail in \S \ref{subsubsec:MainSDE}, we prove existence, uniqueness, and the strong Markov property for weak solutions, $\widehat X$, to a degenerate stochastic differential equation with unbounded coefficients,
\begin{equation}
\label{eq:MimickingSDE}
\begin{aligned}
d\widehat X(t) &= b(t,\widehat X(t)) dt +\sigma(t,\widehat X(t)) d\widehat W(t),\quad t \geq s,
\\
\widehat X(s) &=x.
\end{aligned}
\end{equation}
when the coefficient $\sigma$ is a square root of the coefficient matrix $x_da$ in $\sA_t$ in \eqref{eq:MartingaleGenerator}, that is, when $\sigma\sigma^* = x_da$ on $\HH_T$, where $\HH_T := (0,T) \times \HH$ is the open half-cylinder with $0<T<\infty$. Lastly, suppose we are given a degenerate It\^o process,
\begin{equation}
\label{eq:ItoProcess}
\begin{aligned}
dX(t) &= \beta(t) dt +\xi(t) d W(t),\quad t \geq 0,\\
X(0) &=x,
\end{aligned}
\end{equation}
whose possibly unbounded coefficients $(\xi,\beta)$ are related to those of \eqref{eq:MimickingSDE} as explained in \S \ref{subsubsec:MainMimickingTheorem}. When the coefficients $(\sigma,b)$ in \eqref{eq:MimickingSDE} are determined by the coefficients $(\xi,\beta)$ in \eqref{eq:ItoProcess} as described in \S \ref{subsubsec:MainMimickingTheorem}, we show that the weak solution $\widehat X$ to \eqref{eq:MimickingSDE} ``mimics'' the It\^o process \eqref{eq:ItoProcess} in the sense that $\widehat X(t)$ has the same one-dimensional marginal probability distributions as $X(t)$, for all $t \geq 0$ if $\widehat X(0) = X(0) \in \overline\HH$. Our mimicking theorem complements that of I. Gy\"ongy \cite{Gyongy}, who assumes that \eqref{eq:MimickingSDE} is non-degenerate with bounded, measurable coefficients, that of G. Brunick and S. E. Shreve \cite{BrunickThesis, Brunick_Shreve_2010}, who allow \eqref{eq:MimickingSDE} to be degenerate with unbounded, measurable coefficients, and those of A. Bentata and R. Cont \cite{Bentata_Cont_mimicking} and M. Shi and
J. Wang \cite{ShiThesis, WangThesis} who prove mimicking theorems for a discontinuous semimartingale process with a non-degenerate diffusion component and bounded coefficients.

\subsection{Summary of main results}
We describe our results outlined in the preamble to \S \ref{sec:Intro}.

\subsubsection{Existence and uniqueness of solutions to the martingale problem for a degenerate-elliptic operator with unbounded coefficients}
\label{subsubsec:MainMartingleProblem}
We define an analogue of the classical martingale problem (\cite[p.~138]{Stroock_Varadhan}, \cite[Definition 5.4.5 \& 5.4.10]{KaratzasShreve1991}) when $\RR^d$ is replaced by the closed half-space $\overline{\HH}$.

For $x,y\in\RR$, we denote $x\wedge y = \min\{x,y\}$, $x\vee y = \max\{x,y\}$, $x^+=\max\{x,0\}$, and $x^-=\min\{x,0\}$. The space $C_{\loc}([0,\infty);\overline\HH)$ of continuous functions, $u :[0,\infty)\rightarrow \overline\HH$, endowed with the topology of uniform convergence on compact sets is a complete, separable, metric space. We denote by $\sB(C_{\loc}([0,\infty);\overline\HH))$ the Borel $\sigma$-algebra induced by this topology. As in \cite[Problem 2.4.2]{KaratzasShreve1991}, we see that $\sB(C_{\loc}([0,\infty);\overline\HH))$ is also the $\sigma$-algebra generated by the cylinder sets \eqref{eq:FiltrationCylinderSets}. Following \cite[Problem 2.4.2, Equation (5.3.19) \& Remark 5.4.16]{KaratzasShreve1991}, we consider the filtration $\{\sB_t(C_{\loc}([0,\infty);\overline\HH))\}_{t\geq 0}$ given by
\begin{equation}
\label{eq:FiltrationCylinderSets}
\sB_t(C_{\loc}([0,\infty);\overline\HH)) := \varphi_t \left(\sB(C_{\loc}([0,\infty);\overline\HH))\right), \quad \forall\, t \geq 0,
\end{equation}
where $\varphi_t:C_{\loc}([0,\infty);\overline\HH)\rightarrow C_{\loc}([0,\infty);\overline\HH)$ is defined by
\[
\varphi_t(\omega) := \omega(t\wedge \cdot), \quad \forall\, \omega \in C_{\loc}([0,\infty);\overline\HH).
\]
We then have the

\begin{defn}[Solution to a martingale problem for an operator on a half-space]
\label{defn:Martingale_Problem}
Given $(s,x) \in [0,\infty)\times\overline\HH$, a probability measure $\widehat\PP^{s,x}$ on
$$
(C_{\loc}([0,\infty);\overline\HH),\mathscr{B}(C_{\loc}([0,\infty);\overline\HH))
$$
is a \emph{solution to the martingale problem associated to $\sA_t$ in \eqref{eq:MartingaleGenerator} starting from $(s,x)$} if,
for every $v\in C^{2}_0(\overline\HH)$,
$$
M^v_t(\omega) := v(\omega(t)) - v(\omega(s)) - \int_s^t \sA_u v(\omega(u))\,du,\quad t\geq s,\ \omega \in C_{\loc}([0,\infty);\overline\HH),
$$
is a continuous $\widehat\PP^{s,x}$-martingale with respect to the filtration $\sG_{t+}$, where $\sG_t$ is the augmentation (in the sense of \cite[Definition 2.7.2. (ii)]{KaratzasShreve1991}) under $\widehat\PP^{s,x}$ of the filtration $\{\sB_t(C_{\loc}([0,\infty);\overline\HH)\}_{t\geq 0}$, and $\sG_{t+}$ is its right-continuous version (in the sense of \cite[p. 89]{KaratzasShreve1991}), and $\widehat \PP^{s,x}$ obeys the initial condition,
\begin{equation}
\label{eq:Initial_Condition_Martingale_Problem}
\widehat \PP^{s,x}\left(\omega \in C_{\loc}([0,\infty); \overline\HH): \omega(t)=x, 0\leq t\leq s\right)=1.
\end{equation}
\end{defn}

We note that $\sG_{t+}$ in Definition \ref{defn:Martingale_Problem} satisfies the usual conditions \cite[Definition 1.2.25]{KaratzasShreve1991}.

\begin{rmk}[Reduction to usual filtration]
\label{rmk:Reduction_to_usual_filtration}
\cite[Remark 5.4.16]{KaratzasShreve1991}
By modifying the statement and solution to \cite[Problem 5.4.13]{KaratzasShreve1991} (that is, replacing $\RR^d$ by $\HH$), we see that if $M^v_t$ is a martingale with respect to the filtration $\{\sB_t(C_{\loc}([0,\infty);\overline\HH)\}_{t\geq 0}$, then it is a martingale with respect to the enlarged filtration $\sG_{t+}$ in Definition \ref{defn:Martingale_Problem}.
\end{rmk}

\begin{thm}[Existence and uniqueness of solutions to the martingale problem for a degenerate-elliptic operator with unbounded coefficients]
\label{thm:MainExistenceUniquenessMartProb}
Suppose the coefficients $(a, b)$ in \eqref{eq:MartingaleGenerator} obey the conditions in Assumption \ref{assump:Coeff}. Then, for any $(s,x)\in[0,\infty)\times\overline{\HH}$, there is a unique solution, $\widehat \PP^{s,x}$, to the martingale problem  associated to $\sA_t$ in \eqref{eq:MartingaleGenerator} starting from $(s,x)$.
\end{thm}

When the initial condition $(s,x)$ is clear from the context, we write $\widehat \PP$ instead of $\widehat \PP^{s,x}$. For brevity, when the initial condition is $(0,x)$, we sometimes write $\widehat \PP^{x}$ instead of $\widehat \PP^{0,x}$.

\begin{rmk}[Well-posedness of the classical martingale problem in \cite{Stroock_Varadhan})]
Standard results which ensure \emph{existence} of solutions to the classical martingale problem associated with $\sA_t$
\begin{equation}
\label{eq:ClassicalMartingaleProblemGenerator}
\sA_tv(t,x) := \frac{1}{2} \sum_{i,j=1}^d \alpha_{ij}(t,x) v_{x_ix_j}(x) + \sum_{i=1}^d b_i(t, x)v_{x_i}(x),
\quad (t,x) \in [0,\infty)\times\RR^d,
\end{equation}
require that the coefficients
\begin{equation}
\label{eq:ClassicalMartingaleProblemCoeff}
\begin{aligned}
\alpha :[0,\infty)\times\RR^d &\rightarrow \SSS^d,
\\
b: [0,\infty)\times\RR^d &\rightarrow \RR^d.
\end{aligned}
\end{equation}
be bounded and continuous \cite[Theorem 5.4.22]{KaratzasShreve1991}, \cite[Theorem 6.1.7]{Stroock_Varadhan}; here, $\SSS^d \subset \RR^{d\times d}$ denotes the closed, convex subset of \emph{non-negative} definite, symmetric matrices and $\alpha = (\alpha_{ij})$, $b = (b_i)$. Standard results which ensure \emph{uniqueness} of solutions require, in addition, that the coefficients $(\alpha, b)$ are H\"older continuous and that the matrix $\alpha$ is uniformly elliptic (see \cite[Theorem 5.4.28, Corollary 5.4.29, and Remark 5.4.29]{KaratzasShreve1991} for the time-homogeneous martingale problem). Strict ellipticity of the second-order coefficients matrix is required for the uniqueness of the martingale problem to hold, as in \cite[Theorem 7.2.1]{Stroock_Varadhan}.
\end{rmk}

\begin{rmk}[Approaches to proving uniqueness in the classical martingale problem]
Strong \emph{uniqueness} of solutions to stochastic differential equations as \eqref{eq:MimickingSDE} is guaranteed when the coefficients, $b(t,x)$ and $\sigma(t,x)$, are locally Lipschitz continuous in the spatial variable \cite[Theorem 5.2.5]{KaratzasShreve1991}. We also recall the result of Yamada and Watanabe that pathwise uniqueness of weak solutions implies uniqueness in the sense of probability law \cite[Proposition 5.3.20]{KaratzasShreve1991}. Our article is closer in spirit to a third approach to proving uniqueness of solutions to the classical martingale problem \cite[\S 5.4]{KaratzasShreve1991} which consists in   proving \emph{existence} of solutions in $C([0,T]\times\RR^d)\cap C^{2}((0,T)\times\RR^d)$ to the terminal value problem for the parabolic partial differential equation,
\begin{equation*}
\begin{cases}
u_t + \sA_t u = 0 & \hbox{ on } (0,T)\times\RR^d,
\\
u(T,\cdot) = g &\hbox{ on } \RR^d,
\end{cases}
\end{equation*}
where $g \in C^{\infty}_0(\RR^d)$ and $\sA_t$ is given by \eqref{eq:ClassicalMartingaleProblemGenerator}. Here, $C^{2}((0,T)\times\RR^d)$ denotes the space of continuous functions, $u$, such that $u_t$, $u_{x_i}$ and $u_{x_ix_j}$ are also continuous on $(0,T)\times\RR^d$.
\end{rmk}

\begin{rmk}[Comments on uniqueness]
While \cite[Remark 5.4.31]{KaratzasShreve1991} might appear to provide a simple solution to the uniqueness property asserted by Theorem \ref{thm:MainExistenceUniquenessMartProb} when the nonnegative definite matrix-valued function $x_da$ is in $C^2(\HH;\SSS^d)$, that is not the case. Although we might extend the coefficient, $x_da$, as a nonnegative definite matrix-valued function $x_d^+a$ or $|x_d|a$ in $C^{0,1}(\RR^d;\SSS^d)$, such extensions are not in $C^2(\RR^d;\SSS^d)$, as required by \cite[Remark 5.4.31]{KaratzasShreve1991}.
\end{rmk}

\begin{rmk}[Comments on the regularity of the coefficient matrix, $a$]
Our proof of Theorem \ref{thm:MainExistenceUniquenessMartProb} involves an appeal to \cite[Lemma 6.1.1]{FriedmanSDE} to find a square root, $\varsigma \in C_{\loc}([0,\infty)\times\overline\HH;\RR^{d\times d})$, such that $\varsigma\varsigma^*=a$ on $[0,\infty)\times\overline\HH$. That appeal is valid since $a\in C_{\loc}([0,\infty)\times\overline\HH;\SSS^d)$ and is strictly elliptic on $[0,\infty)\times\overline\HH$. More generally, if $a$ is $C^{m,\alpha}$ (respectively, $C^m$), for an integer $m\geq 0$ and $\alpha\in (0,1]$, then one may choose a square root, $\varsigma$, which is also $C^{m,\alpha}$ (respectively, $C^m$). If the matrix, $a(t,x)$, is merely non-negative for all $(t,x) \in [0,\infty)\times\overline\HH$ but $a$ is $C^2$ on $[0,\infty)\times\overline\HH$, then one may choose a square root, $\varsigma$, which is Lipschitz on $[0,\infty)\times\overline\HH$, according to \cite[Theorem 6.1.2]{FriedmanSDE}.
\end{rmk}

\subsubsection{Existence and uniqueness of weak solutions to a degenerate stochastic differential equation with unbounded coefficients}
\label{subsubsec:MainSDE}
We start by describing the conditions imposed on the coefficient functions $(\sigma,b)$ which define a degenerate stochastic differential equation \eqref{eq:MimickingSDE}.

\begin{assump}[Properties of the coefficients of the stochastic differential equation]
\label{assump:MimickingCoeffSpecialForm}
The coefficient functions $(\sigma, b)$ in \eqref{eq:MimickingSDE} obey the following conditions.
\begin{enumerate}
\item There is a function $\varsigma\in C_{\loc}([0,\infty)\times\overline\HH;\RR^{d\times d})$ such that
\begin{equation}
\label{eq:FormOfDiffusionMatrix}
\sigma(t,x) = \sqrt{x_d} \varsigma(t,x),\quad\forall\, (t,x)\in [0,\infty)\times\overline\HH.
\end{equation}
\item If we define $a: [0,\infty)\times\overline\HH \to \SSS^d$ by
\begin{equation}
\label{eq:Definition_a}
a(t,x) := \varsigma(t,x)\varsigma^*(t,x), \quad \forall\, (t,x)\in [0,\infty)\times\overline\HH,
\end{equation}
then the coefficient functions $(a, b)$ obey the conditions in Assumption \ref{assump:Coeff}.
\end{enumerate}
\end{assump}

The constraints on the coefficients $(\sigma,b)$ implied by Assumption \ref{assump:MimickingCoeffSpecialForm} are mild enough that they include many examples of interest in mathematical finance.

\begin{exmp}[Heston stochastic differential equation]
\label{exmp:HestonSDE}
The conditions in Assumption \ref{assump:MimickingCoeffSpecialForm} are obeyed by the coefficients of the $\RR^2$-valued log-Heston process \cite{Heston1993} with killing,
\begin{equation}
\label{eq:Heston_Process}
\begin{aligned}
dX_1(t) &= \left(r-q-\frac{1}{2}X_2(t)\right) dt + \sqrt{X_2(t)} dW_1(t),
\\
dX_2(t) &= \kappa(\theta-X_2(t)) dt + \zeta \sqrt{X_2(t)}\left(\varrho dW_1(t)+\sqrt{1-\varrho^2}dW_2(t)\right),
\end{aligned}
\end{equation}
where $q\in\RR$, $r\geq 0$, $\kappa>0$, $\theta>0$, $\zeta\neq 0$, and $\varrho\in (-1,1)$ are constants.
\end{exmp}

\begin{exmp}[Parabolic Heston partial differential equation]
\label{exmp:HestonPDE}
The conditions in Assumption \ref{assump:Coeff} are obeyed by the coefficients of the parabolic Heston partial differential operator,
\begin{equation}
\label{eq:HestonPDE}
-Lu = -u_t + \frac{x_2}{2}\left(u_{x_1x_2} + 2\varrho\zeta u_{x_1x_2} + \zeta^2 u_{x_2x_2}\right) + \left(r-q-\frac{x_2}{2}\right)u_{x_1} + \kappa(\theta-x_2)u_{x_2} - ru,
\end{equation}
where the coefficients are as in Example \ref{exmp:HestonSDE}.
\end{exmp}

The following theorem does not follow by the classical results \cite[Theorem 5.3.3, Theorem 4.4.2 and Corollary 4.4.3]{Ethier_Kurtz} because our operator is not time-homogeneous as is required in the hypotheses of the cited results.

\begin{thm}[Existence, uniqueness, and strong Markov property of weak solutions to a degenerate stochastic differential equation with unbounded coefficients]
\label{thm:MainWeakExistenceUniquenessSDE}
Suppose that the coefficients $(\sigma, b)$ in \eqref{eq:MimickingSDE} obey the conditions in Assumption \ref{assump:MimickingCoeffSpecialForm}. Let $(s,x)\in [0,\infty)\times\overline{\HH}$. Then,
\begin{enumerate}
\item There is a weak solution, $(\widehat X, \widehat W)$, $(\Omega, \sF, \PP)$, $\{\sF_t\}_{t \geq s}$, to the stochastic differential equation \eqref{eq:MimickingSDE} such that $\widehat X(s)=x$, $\PP$-a.s.
\item The weak solution is unique in the sense of probability law, that is, if
$$
(\widehat X^i, \widehat W^i), (\Omega^i, \sF^i, \PP^i), (\sF^i_t)_{t \geq s}, \quad i=1,2,
$$
are two weak solutions to the stochastic differential equation \eqref{eq:MimickingSDE} started at $x$ at time $s$, then the two processes $X^1$ and $X^2$ have the same law.
\item The unique weak solution, $(\widehat X, \widehat W)$, $(\Omega, \sF, \PP)$, $\{\sF_t\}_{t \geq s}$, has the strong Markov property.
\end{enumerate}
\end{thm}

When the initial condition $(s,x)$ is not clear from the context, we write $X^{s,x}$ instead of $X$. For brevity, when the initial condition is $(0,x)$, we sometimes write $X^x$ instead of $X^{0,x}$ or $X$.

\begin{rmk}[Non-exploding solutions]
In the one-dimensional case, \cite[Remark 5.5.19]{KaratzasShreve1991} can be applied to show that solutions to \eqref{eq:MimickingSDE} are non-exploding; \cite[Theorem 5.3.10]{Ethier_Kurtz} may also be applied to give this conclusion, noting that the moments of order $2m$ ($m\geq 1$) are bounded via \eqref{eq:ItoLemma6}.
\end{rmk}

\subsubsection{Mimicking one-dimensional marginal probability distributions of a degenerate It\^o process with unbounded coefficients}
\label{subsubsec:MainMimickingTheorem}
Let $X$ be an $\RR^d$-valued It\^o process as in \eqref{eq:ItoProcess}, where $W$ is an $\RR^r$-valued Brownian motion on a filtered probability space, $(\Omega, \sF, \PP, \{\sF_t\}_{t\geq 0})$, satisfying the usual conditions \cite[Definition 1.2.25]{KaratzasShreve1991}, $\beta$ is an $\RR^d$-valued, adapted process, and $\xi$ is a $\RR^{d\times r}$-valued, adapted process satisfying the integrability condition,
\begin{equation}
\label{eq:IntegrabilityCondition}
\EE\left[\int_0^t\left(|\beta(s)|+|\xi(s)\xi^*(s)|\right) \,ds\right] < \infty, \quad \forall\, t \geq 0.
\end{equation}
We assume that $X(0)$ is non-random, with
\begin{equation}
\label{eq:Initial_cond_Ito_process}
X(0)=x_0\in\overline{\HH},
\end{equation}
and that for all $ t\geq 0$ we have
\begin{equation}
\label{eq:SupportItoProcess}
X(t) \in \overline\HH, \quad \PP\hbox{-a.s.},
\end{equation}
that is, $\overline\HH$ is a state space for the process, $X(t)$.

We can now state the main result of this article which is the analogue of \cite[Theorem 4.6]{Gyongy} for but for an It\^o process \eqref{eq:ItoProcess} with degenerate, unbounded diffusion coefficient, $\xi(t)$, and possibly unbounded drift coefficient, $\beta(t)$.

\begin{thm}[Mimicking theorem for degenerate It\^o processes with unbounded coefficients]
\label{thm:MainMarginalsMatching}
Let $X(t)$ be an It\^o process as in \eqref{eq:ItoProcess}, with coefficients $\xi(t)$ and $\beta(t)$, and which obeys \eqref{eq:SupportItoProcess}. We define deterministic functions, measurable with respect to the Lebesque measure  on $[0,\infty)\times\HH$, by
\begin{align}
\label{eq:DefinitionMimickingCoeff_b}
b(t, x) &:= \EE\left[\beta(t)|X(t)=x\right],
\\
\label{eq:DefinitionMimickingCoeff_a}
x_d a(t, x) &:= \EE\left[\xi(t)\xi^*(t)|X(t)=x\right],
\end{align}
and assume that the coefficient functions, $(a,b)$, satisfy Assumption \ref{assump:Coeff}. If $\sigma$ is a coefficient function obeying \eqref{eq:FormOfDiffusionMatrix} and \eqref{eq:Definition_a}, then the unique weak solution, $\widehat X$, to the stochastic differential equation \eqref{eq:MimickingSDE}, with initial condition $\widehat X(0)=X(0)=x_0\in \bar\HH$, given by Theorem \ref{thm:MainWeakExistenceUniquenessSDE}, has the same one-dimensional marginal distributions as $X$.
\end{thm}

\begin{rmk}[Mimicking stochastic differential equation]
We call \eqref{eq:MimickingSDE} the \emph{mimicking stochastic differential equation} defined by the It\^o process \eqref{eq:ItoProcess} when its coefficients $(\sigma, b)$ are as in the statement of Theorem \ref{thm:MainMarginalsMatching}.
\end{rmk}

\subsection{Connections with previous research on martingale and mimicking problems}
We briefly survey previous work on uniqueness of solutions to the martingale problem for degenerate differential operators, uniqueness and the strong Markov property for solutions to degenerate stochastic differential equations, and mimicking problems.

\subsubsection{Mimicking theorems}
Gy\"ongy \cite[Theorem 4.6]{Gyongy} proves existence of a mimicking process as in Theorem \ref{thm:MainMarginalsMatching} --- although not the uniqueness or strong Markov properties --- with conditions on the coefficients $(\sigma,b)$ which are both partly \emph{weaker} than those of Theorem \ref{thm:MainMarginalsMatching}, because the functions $b:[0,\infty)\times\RR^d\to\RR^d$ and $\sigma:[0,\infty)\times\RR^d\to\RR^{d\times d}$ are only required to be Borel-measurable, but also partly \emph{stronger} than those of Theorem \ref{thm:MainMarginalsMatching}, because the functions $(\sigma,b)$ are required to be uniformly bounded on $[0,\infty)\times\RR^d$ and $\sigma\sigma^*$ is required to be uniformly positive definite on $[0,\infty)\times\RR^d$. Since Gy\"ongy only requires that the coefficients $(\sigma, b)$ of the corresponding mimicking stochastic differential equation \eqref{eq:MimickingSDE} are Borel measurable functions, he uses an auxiliary regularizing procedure to construct a weak solution $\widehat X$ to \eqref{eq:MimickingSDE}. Uniqueness of the weak solution is not proved under the hypotheses of \cite[Theorem 4.6]{Gyongy} and the main obstacle here is the lack of regularity of the coefficients
$(\sigma, b)$.

The hypotheses of \cite[Theorem 4.6]{Gyongy} are quite restrictive, as we can see that they would exclude a process, $X$, such as that in Example \ref{exmp:HestonSDE}, even though the coefficients of its mimicking processes, $\widehat X$, can be found by explicit calculation \cite{Atlan_2006} (see also \cite{Antonov_Misirpashaev_Piterbarg_2009}). Moreover, N. Nadirashvili shows \cite{Nadirashvili_1997} that uniqueness of stochastic differential equations with \emph{measurable} coefficients satisfying the assumptions of non-degeneracy and boundedness in \cite[Theorem 4.6]{Gyongy} does not hold in general when $d \geq 3$.

Brunick and Shreve \cite[Corollary 2.16]{BrunickThesis}, \cite{Brunick_Shreve_2010} prove an extension of \cite[Theorem 4.6]{Gyongy} which relaxes the requirements that $\sigma\sigma^*$ be uniformly positive definite on $[0,\infty)\times\RR^d$ and that the functions $\sigma$ and $b$ are bounded on $[0,\infty)\times\RR^d$. Moreover, they significantly extend Gy\"ongy's theorem \cite{Gyongy} by replacing the non-degeneracy and boundedness conditions on the coefficients of the It\^o process, $X$, by a mild integrability condition \eqref{eq:IntegrabilityCondition}. Using purely probabilistic methods, they show existence of weak solutions to  stochastic differential equations of diffusion type which preserve not only the one-dimensional marginal distributions of the It\^o process, but also certain statistics, such as the running maximum or average of one of the components. More recently,  Brunick \cite{Brunick_2013} establishes weak uniqueness for a degenerate stochastic differential equation with applications to
pricing Asian options.

Bentata and Cont \cite{Bentata_Cont_mimicking} and Shi and Wang \cite{ShiThesis, WangThesis} extend Gy\"ongy's mimicking theorem to \emph{discontinuous}, non-degenerate semimartingales. Under assumptions of continuity and boundedness on the coefficients of the process and non-degeneracy condition of the diffusion matrix or of the Lev\'y operator, they prove uniqueness of solutions to the forward Kolmogorov equation associated with the generator of the mimicking process. In this setting, they show that weak uniqueness to the mimicking stochastic differential equation holds and that the mimicking process satisfies the Markov property.

M. Atlan \cite{Atlan_2006} obtains closed-form solutions for the mimicking coefficients when the It\^o process $X$ has the form in Example \ref{exmp:HestonSDE}, where the volatility modeled by a Bessel or Cox-Ingersoll-Ross (that is, a Feller square root) process. This is possible because explicit, tractable expressions are known for the distribution of Bessel processes and the Cox-Ingersoll-Ross process can be obtained from a Bessel process by a suitable transformation (the Cox-Ingersoll-Ross process is a deterministic, time-changed Bessel process, multiplied by a deterministic function of time \cite[Lemma 2.4]{Atlan_2006}).

\subsubsection{Uniqueness and strong Markov property of solutions to degenerate stochastic differential equations}
There are counterexamples to uniqueness of weak solutions to degenerate stochastic differential equations such as \eqref{eq:MimickingSDE}; see \cite{Cherny_Engelbert_2005}, \cite[\S 5.5]{KaratzasShreve1991}, and \cite{Rutkowski_1990}. Moreover, as noted in \cite[Example 2.2.10]{BrunickThesis}, the Markov property of solutions to \eqref{eq:MimickingSDE} is not guaranteed; see also \cite[Example 3.10]{Cherny_Engelbert_2005} for another example of non-Markov process arising as the solution to a one-dimensional stochastic differential equation. Sufficient conditions for weak solutions of degenerate stochastic differential equations such as \eqref{eq:MimickingSDE} to be Markov are provided by \cite[Theorem 7.1.2]{Oksendal_2003}, the combination \cite[Theorems 5.2.9 \& 5.4.20]{KaratzasShreve1991}, \cite{Krylov_ControlledDiffProc} and elsewhere.

\subsubsection{Uniqueness of solutions to degenerate martingale problems}
Uniqueness for solutions to the classical martingale problem (\cite[p.~138]{Stroock_Varadhan}, \cite[Definition 5.4.5 \& 5.4.10]{KaratzasShreve1991}), for suitable coefficients $(a,b)$, is proved in \cite[Chapter 7]{Stroock_Varadhan} (see \cite[\S 7.0]{Stroock_Varadhan} for a comprehensive outline), via uniqueness of solutions to a certain Kolmogorov backward equation. Special cases of uniqueness for solutions to the martingale problem are established in \cite[Theorem 6.3.4]{Stroock_Varadhan} (via uniqueness of weak solutions to a stochastic differential equation in \cite[Theorem 5.3.2]{Stroock_Varadhan}), \cite[Corollary 6.3.3]{Stroock_Varadhan} (via existence and uniqueness of solutions to a parabolic partial differential equation in \cite[Theorem 3.2.6]{Stroock_Varadhan}); as Stroock and Varadhan  observe \cite[\S 6.3]{Stroock_Varadhan}, their special cases do not cover situations where $\sigma\sigma^*$ is only non-negative definite.
Their general uniqueness result \cite[\S 7.0]{Stroock_Varadhan} does \emph{not} apply to the differential operator in Example \ref{exmp:HestonPDE} or other differential operators with similar degeneracies. Similarly, while uniqueness results for solutions to the martingale problem for certain degenerate elliptic differential operators is described by S. N. Ethier and T. G. Kurtz in \cite[Theorem 8.2.5]{Ethier_Kurtz}, their results do \emph{not} apply to the differential operator in Example \ref{exmp:HestonPDE} or other differential operators with similar degeneracies. Well-posedness of the martingale problem for certain time-homogeneous, degenerate operators was established in \cite{Athreya_Barlow_Bass_Perkins_2002}, \cite{Bass_Perkins_2003}, \cite{Bass_Perkins_2004} and \cite{Bass_Lavrentiev_2007}, but their results do not apply to our operator \eqref{eq:MartingaleGenerator} under the hypotheses of our Theorem \ref{thm:MainExistenceUniquenessMartProb}.

The following example of Stroock and Varadhan shows that solutions to degenerate martingale problems can easily fail to be unique.

\begin{exmp}[Non-uniqueness of solutions to certain degenerate martingale problems]
\cite[Exercise 6.7.7]{Stroock_Varadhan}
\label{exmp:StroockVaradhan}
Consider the one-dimensional generator, $\sA u(x)=(|x|^{\alpha}\wedge 1) u''(x)$, for $u \in C^2(\RR)$, with $0<\alpha<1$. The operator $\sA$ is degenerate at $x=0$ and uniqueness in law for solutions to the martingale problem for $\sA$ fails. See \cite{Engelbert_Schmidt_1984} and \cite[\S 5.5]{KaratzasShreve1991} for additional details. \qed
\end{exmp}

Additional examples of non-uniqueness of solutions to the (sub-)martingale problem are provided by R. F. Bass and A. Lavrentiev \cite{Bass_Lavrentiev_2007}, along with suitable boundary conditions designed to achieve uniqueness.

The well-known \emph{Yamada criterion} \cite[p. 115]{Yamada_1978} can be used to provide existence and uniqueness of strong solutions to one-dimensional stochastic differential equations with non-Lipschitz coefficients \cite[p. 117]{Yamada_1978}. A simple generalization was noted long ago by N. Ikeda and S. Watanabe \cite[Theorem IV.3.2 and footnote 1, p. 182]{Ikeda_Watanabe}, for coefficients $\sigma: \RR_+\times \RR^d \to \RR^{d\times r}$ and $b: \RR_+\times \RR^d \to \RR^d$ of a stochastic differential equation such as \eqref{eq:MimickingSDE} with $d>1$ and $r=1$. However, the case $d>1$ and $r>1$ is considerably more difficult. While there are more substantial generalizations of Yamada's theorem due to Luo \cite{Luo_2011}, Altay and Schmock \cite{Altay_Schmock_2012}, and references cited therein and though V. I. Bogachev, N. V. Krylov, M. R\"ockner, and X. Zhang \cite{Bogachev_Krylov_Rockner_2009, Rockner_Zhang_2010} also provide related uniqueness results, they do not cover the situation to which our Theorem \ref{thm:MainWeakExistenceUniquenessSDE} applies.

\subsection{Future research}
It would be useful to establish sufficient conditions on the coefficients of the It\^o process, $X$, which would ensure that our Assumption \ref{assump:Coeff} on the mimicking coefficients is satisfied, or relax these assumptions further.

Because the coefficients of the stochastic differential equation \eqref{eq:MimickingSDE} are H\"older continuous, it is natural to ask whether pathwise uniqueness for the weak solutions to the mimicking stochastic differential equation holds and so conclude that the weak solutions are actually strong. Positive answers to this question for certain degenerate stochastic differential equations are obtained by R. F. Bass, K. Burdzy and Z. Q. Chen \cite{Bass_Burdzy_Chen_2007}.


Given an arbitrary It\^o process with coefficients $(\xi,\beta)$, it is difficult to determine, in general, whether the coefficients $(\sigma, b)$ of the mimicking stochastic differential equation possess any further regularity than measurability. Bentata and Cont \cite{Bentata_Cont_mimicking} assume that the coefficients of the mimicking process are continuous and they provide a sufficient condition under which this assumption is satisfied. Specifically, if the It\^o process is a one-dimensional process given in the form
\[
X(t) = f(Z(t)), \quad \forall\, t \geq 0,
\]
where $f : \RR^d\rightarrow\RR $ is a $C^2$ function with bounded derivatives, $f_{x_d} \neq 0$, and $Z$ is a ``nice'', $\RR^d$-valued Markov process, then the mimicking coefficients are continuous functions. This construction is useful when one wants to reduce the dimensionality of a Markov process. In our future work, we hope to relax the local H\"older regularity condition on the mimicking coefficients and the conditions under which one can still recover the weak uniqueness of solutions.

We also plan to explore to what extent our techniques can be used in the presence of other types of degeneracy occurring in the diffusion matrix, for example, $\bar a_{ij}(t,x) = x_d^\alpha  a_{ij}(t,x)$, with $\alpha \neq 1$, where $a(t,x)$ is uniformly elliptic. By Example \ref{exmp:StroockVaradhan}, we expect that weak uniqueness will not hold for arbitrary values of $\alpha$, and then one may consider the question of Markovian selection of a weak solution that mimics the one-dimensional marginal distributions of the It\^o process.

\subsection{Outline of the article}
\label{subsec:Guide}
In \S \ref{sec:WeightedHolderSpaces}, we define the H\"older spaces required to prove Theorem \ref{thm:MainExistenceUniquenessPDE} (existence and uniqueness of solutions to a degenerate-parabolic partial differential equation on a half-space with unbounded coefficients) and provide a detailed description of the conditions required of the coefficients $(a,b,c)$ in the statement of Theorem \ref{thm:MainExistenceUniquenessPDE}. Section \ref{sec:MimickingTheorem} contains the proofs of Theorems \ref{thm:MainExistenceUniquenessMartProb}, \ref{thm:MainWeakExistenceUniquenessSDE}, and \ref{thm:MainMarginalsMatching}. In \S \ref{subsec:ExistenceSDE}, we prove existence of solutions to the degenerate martingale problem and degenerate stochastic differential equation specified in Theorems \ref{thm:MainExistenceUniquenessMartProb} and \ref{thm:MainWeakExistenceUniquenessSDE}, while in \S \ref{subsec:UniquenessSDE}, we
prove uniqueness and the strong Markov property in Theorems \ref{thm:MainExistenceUniquenessMartProb} and \ref{thm:MainWeakExistenceUniquenessSDE}. Lastly, in \S \ref{subsec:MatchingOneDimMarginals}, we prove our mimicking theorem for a degenerate It\^o process, namely, Theorem \ref{thm:MainMarginalsMatching}.

\subsection{Acknowledgments} We are very grateful to Gerard Brunick and Mihai S\^irbu for useful conversations and suggestions regarding the mimicking theorem and also to the anonymous referee for a careful reading of our manuscript, comments and suggestions, corrections, and an alternative, shorter proof of Proposition \ref{prop:ItoLemma}.

\section{Weighted H\"older spaces and coefficients of the differential operators}
\label{sec:WeightedHolderSpaces}
In \S \ref{subsec:DHKHolderSpaces}, we introduce the H\"older spaces required for the statement and proof of Theorem \ref{thm:MainExistenceUniquenessPDE}, while in \S \ref{subsec:AssumptionsPDECoefficients}, we describe the regularity and growth conditions required of the coefficients $(a, b, c)$ in Theorem \ref{thm:MainExistenceUniquenessPDE}. In \S \ref{subsec:MainExistenceUniquenessPDE}, we recall the statement of our previous result,
Theorem \ref{thm:MainExistenceUniquenessPDE} in \cite{Feehan_Pop_mimickingdegen_pde}, which establishes existence and uniqueness of solutions to the Cauchy problem for the parabolic operator $-\partial_t+\sA_t$.

\subsection{Weighted H\"older spaces}
\label{subsec:DHKHolderSpaces}
For $a>0$, we denote
$$
\HH_{a,T} := (0,T) \times\RR^{d-1}\times(0,a),
$$
and, when $T=\infty$, we denote $\HH_\infty = (0,\infty) \times\HH$ and $\HH_{a,\infty} = (0,\infty)\times\RR^{d-1}\times(0,a)$. We denote the usual closures these half-spaces and cylinders by $\overline{\HH}:=\RR^{d-1}\times[0,\infty)$, $\overline{\HH}_T:=[0,T]\times\overline\HH$, while $\overline{\HH}_{a,T}:=[0,T]\times\RR^{d-1}\times[0,a]$. We write points in $\HH$ as $x:=(x',x_d)$, where $x':=(x_1, x_2, \ldots, x_{d-1})\in \RR^{d-1}$. For $x^0 \in \overline{\HH}$ and $R>0$, we let
\begin{align*}
B_R(x^0) &:= \left\{x \in \HH: |x-x^0| < R \right\},
\\
Q_{R,T}(x^0) &:= (0,T)\times B_R(x^0),
\end{align*}
and denote their usual closures by $\bar B_R(x^0) := \{x \in \HH: |x-x^0| \leq R\}$ and $\bar Q_{R,T}(x^0) := [0,T]\times \bar B_R(x^0)$, respectively. We write $B_R$ or $Q_{R,T}$ when the center, $x^0$, is clear from the context or unimportant.

A parabolic partial differential equation with a degeneracy similar to that considered in this article arises in the study of the porous medium equation \cite{DaskalHamilton1998, Daskalopoulos_Rhee_2003, Koch}. The existence, uniqueness, and regularity theory for such equations is facilitated by the use of H\"older spaces defined by the \emph{cycloidal metric} on $\HH$ introduced by P. Daskalopoulos and R. Hamilton \cite{DaskalHamilton1998} and, independently, by H. Koch \cite{Koch}. See \cite[p. 901]{DaskalHamilton1998} for a discussion of this metric. Following \cite[p. $901$]{DaskalHamilton1998}, we define the \emph{cycloidal distance} between two points, $P_1=(t_1, x^1), P_2=(t_2, x^2) \in [0,\infty)\times\overline{\HH}$, by
\begin{equation}
\label{eq:CycloidalMetric}
\begin{aligned}
s(P_1, P_2) &:= \frac{\sum_{i=1}^d |x_i^1-x_i^2|}{\sqrt{x_d^1}+\sqrt{x_d^2}+\sqrt{\sum_{i=1}^{d-1} |x_i^1-x_i^2|}}+\sqrt{|t_1-t_2|}.
\end{aligned}
\end{equation}
Following \cite[p. 117]{Krylov_LecturesHolder}, we define the usual parabolic distance between points $P_1, P_2 \in [0,\infty)\times\RR^d$ by
\begin{equation}
\label{eq:UsualMetric}
\begin{aligned}
\rho(P_1, P_2) &:= \sum_{i=1}^d |x_i^1-x_i^2| + \sqrt{|t_1-t_2|}.
\end{aligned}
\end{equation}

\begin{rmk}[Equivalence of the cycloidal and
parabolic distance functions on suitable subsets of $[0,\infty)\times\HH$]
\label{rmk:EquivalenceMetric}
The cycloidal and parabolic distance functions, $s$ and $\rho$, are equivalent on sets of the form $[0,\infty)\times\RR^{d-1}\times[y_0,y_1]$, for any $0<y_0<y_1$.
\end{rmk}

Let $\Omega\subset (0,T)\times\HH$ be an open set and $\alpha \in (0,1)$. We denote by $C(\bar \Omega)$ the space of bounded, continuous functions on $\bar \Omega$, and by $C^{\infty}_0(\bar \Omega)$ the space of smooth functions with compact support in $\bar \Omega$. For a function $u:\bar \Omega\rightarrow\RR$, we consider the following norms and seminorms
\begin{align}
\label{eq:SupNorm}
 \|u\|_{C(\bar \Omega)}             &= \sup_{P \in \bar \Omega}  |u(P)|,\\
\label{eq:SeminormS}
 [u]_{C^{\alpha}_s(\bar \Omega)}      &= \sup_{\stackrel{P_1, P_2 \in \bar \Omega,} {P_1 \neq P_2}} \frac{|u(P_1)-u(P_2)|}{s^{\alpha}(P_1,P_2)},\\
\label{eq:SeminormRHO}
 [u]_{C^{\alpha}_{\rho}(\bar \Omega)} &= \sup_{\stackrel{P_1, P_2 \in \bar \Omega,} {P_1 \neq P_2}} \frac{|u(P_1)-u(P_2)|}{\rho^\alpha(P_1,P_2)}.
\end{align}
We say that $u \in C^{\alpha}_s(\bar \Omega)$ if $u \in  C(\bar \Omega)$ and
\[
\|u\|_{C^{\alpha}_s(\bar \Omega)} = \|u\|_{C(\bar \Omega)}+[u]_{C^{\alpha}_s(\bar \Omega)}  < \infty.
\]
Analogously, we define the H\"older space $C^{\alpha}_{\rho}(\bar \Omega)$ of functions $u$ which satisfy
\[
\|u\|_{C^{\alpha}_{\rho}(\bar \Omega)} = \|u\|_{C(\bar \Omega)}+[u]_{C^{\alpha}_{\rho}(\bar \Omega)}  < \infty.
\]
We say that $u \in C^{2+\alpha}_s(\bar \Omega)$ if
\[
\|u\|_{C^{2+\alpha}_s(\bar \Omega)}
:= \|u\|_{C^{\alpha}_s(\bar \Omega)}+\|u_t\|_{C^{\alpha}_s(\bar \Omega)}
+\max_{1\leq i\leq d}\|u_{x_i}\|_{C^{\alpha}_s(\bar \Omega)}
+\max_{1\leq i,j\leq d}\|x_du_{x_ix_j}\|_{C^{\alpha}_s(\bar \Omega)} <\infty,
\]
and $u \in C^{2+\alpha}_\rho(\bar \Omega)$ if
\[
\|u\|_{C^{2+\alpha}_{\rho}(\bar \Omega)}
= \|u\|_{C^{\alpha}_{\rho}(\bar \Omega)}+\|u_t\|_{C^{\alpha}_{\rho}(\bar \Omega)}
+ \max_{1\leq i\leq d} \|u_{x_i}\|_{C^{\alpha}_{\rho}(\bar \Omega)}
+ \max_{1\leq i,j\leq d} \|u_{x_ix_j}\|_{C^{\alpha}_{\rho}(\bar \Omega)}  <\infty.
\]
In other words, a function $u$ belongs to $C^{2+\alpha}_s(\bar\Omega)$ if $u$ and its first-order derivatives are H\"older continuous with respect to the cycloidal distance function, $s$, on the closed set $\bar\Omega$, and the second-order derivatives of $u$ multiplied by the weight, $x_d$, are H\"older continuous with respect to the cycloidal distance function, $s$, on $\bar\Omega$.

We denote by $C^{\alpha}_{s, \loc}(\bar \Omega)$ the space of functions $u$ with the property that for any compact set $K\subseteqq \bar \Omega$, we have $u \in C^{\alpha}_s(K)$. Analogously, we define the spaces $C^{2+\alpha}_{s, \loc}(\bar \Omega)$, $C^{\alpha}_{\rho, \loc}(\bar \Omega)$ and $C^{2+\alpha}_{\rho, \loc}(\bar \Omega)$.

We prove existence, uniqueness and regularity of solutions for a parabolic operator \eqref{eq:Generator} whose second-order coefficients are degenerate on $\partial\HH$. For this purpose, we will make use of the following H\"older spaces,

\begin{equation*}
\begin{aligned}
\sC^{\alpha}(\overline{\HH}_T)
&:= \left\{u: u \in C^{\alpha}_s(\overline{\HH}_{1,T}) \cap C^{\alpha}_{\rho}(\overline{\HH}_T\less\HH_{1,T}) \right\},\\
\sC^{2+\alpha}(\overline{\HH}_T)
&:= \left\{u: u \in C^{2+\alpha}_s(\overline{\HH}_{1,T}) \cap C^{2+\alpha}_{\rho}(\overline{\HH}_T\less\HH_{1,T}) \right\}.
\end{aligned}
\end{equation*}
We define $\sC^{\alpha}(\overline{\HH})$ and $\sC^{2+\alpha}(\overline{\HH})$ in the analogous manner.

The coefficient functions $x_d a_{ij}(t,x)$, $b_i(t,x)$ and $c(t,x)$ of the parabolic operator \eqref{eq:Generator} are allowed to have linear growth in $|x|$. To account for the unboundedness of the coefficients, we  augment our definition of H\"older spaces by introducing weights of the form $(1+|x|)^q$, where $q\geq 0$ will be suitably chosen in the sequel. For $q\geq0$, we define
\begin{equation}
\label{eq:DefinitionX_0_q}
\begin{aligned}
\|u\|_{\sC^0_q(\overline{\HH})}
&:= \sup_{x \in \overline{\HH}}\left(1+ |x|\right)^q |u(x)|,
\end{aligned}
\end{equation}
and, given $T>0$, we define
\begin{align}
\label{eq:DefinitionX_0_q_T}
\|u\|_{\sC^0_q(\overline{\HH}_T)}
&:= \sup_{(t,x)\in\overline{\HH}_T} \left(1+|x|\right)^q|u(t,x)|.
\end{align}
Moreover, given $\alpha\in(0,1)$, we define
\begin{align}
\label{eq:DefinitionX_Alpha_q_T}
\|u\|_{\sC^{\alpha}_q(\overline{\HH}_T)}
&:= \|u\|_{\sC^0_q(\overline{\HH}_T)} + [(1+|x|)^q u]_{C^{\alpha}_s(\overline{\HH}_{1,T})} + [(1+|x|)^q u]_{C^{\alpha}_{\rho}(\overline{\HH}_T\less\HH_{1,T})},\\
\label{eq:DefinitionX_Two_Aplha_q_T}
\|u\|_{\sC^{2+\alpha}_q(\overline{\HH}_T)}
&:=  \|u\|_{\sC^{\alpha}_q(\overline{\HH}_T)}+ \|u_t\|_{\sC^{\alpha}_q(\overline{\HH}_T)} + \|u_{x_i}\|_{\sC^{\alpha}_q(\overline{\HH}_T)}+ \|x_du_{x_ix_j}\|_{\sC^{\alpha}_q(\overline{\HH}_T)}.
\end{align}
The vector spaces
\begin{align*}
\sC^0_q(\overline{\HH}_T)
&:= \left\{u \in C(\overline{\HH}_T): \|u\|_{\sC^{0}_q(\overline{\HH}_T)} <\infty \right\},\\
\sC^{\alpha}_q(\overline{\HH}_T)
&:= \left\{u \in \sC^{\alpha}(\overline{\HH}_T): \|u\|_{\sC^{\alpha}_q(\overline{\HH}_T)} <\infty \right\},\\
\sC^{2+\alpha}_q(\overline{\HH}_T)
&:= \left\{u \in \sC^{2+\alpha}(\overline{\HH}_T): \|u\|_{\sC^{2+\alpha}_q(\overline{\HH}_T)} <\infty \right\},
\end{align*}
can be shown to be Banach spaces with respect to the norms \eqref{eq:DefinitionX_0_q_T}, \eqref{eq:DefinitionX_Alpha_q_T} and \eqref{eq:DefinitionX_Two_Aplha_q_T}, respectively. We define the vector spaces $\sC^0_q(\overline{\HH})$, $\sC^{\alpha}_q(\overline{\HH})$, and $\sC^{2+\alpha}_q(\overline{\HH})$ similarly, and each can be shown to be a Banach space when equipped with the corresponding norm.

We let $\sC^{2+\alpha}_{q, \loc}(\overline{\HH}_T)$ denote the vector space of functions $u$ such that for any compact set $K \subset \overline{\HH}_T$, we have $u \in \sC^{2+\alpha}_q(K)$, for all $q \geq 0$.

When $q=0$, the subscript $q$ is omitted in the preceding definitions.

\subsection{Coefficients of the differential operators}
\label{subsec:AssumptionsPDECoefficients}
In our companion article \cite{Feehan_Pop_mimickingdegen_pde}, we used a Schauder approach to prove existence, uniqueness, and regularity of solutions to the degenerate-parabolic partial differential equation,
\begin{equation}
\label{eq:Problem}
\begin{aligned}
\begin{cases}
Lu=f & \hbox{ on  } \HH_T,\\
u(0,\cdot)=g & \hbox{  on   } \overline{\HH},
\end{cases}
\end{aligned}
\end{equation}
where
\begin{equation}
\label{eq:Generator}
\begin{aligned}
-Lu = -u_t + \sum_{i,j=1}^d x_d a_{ij}u_{x_ix_j} + \sum_{i=1}^d b_iu_{x_i} + cu, \quad \forall\, u \in C^{2}(\HH_T).
\end{aligned}
\end{equation}
Observe that
$$
-Lu = -u_t + \sA_t u + cu,
$$
where $\sA_t$ is given in \eqref{eq:MartingaleGenerator}, provided we absorb the factor $1/2$ into the definition of the coefficients, $a_{ij}$. Unless other conditions are explicitly substituted, we require in this article that the coefficients $(a,b,c)$ of the parabolic differential operator $L$ in \eqref{eq:Generator} satisfy the conditions in the following

\begin{assump}[Properties of the coefficients of the parabolic differential operator]
\label{assump:Coeff}
There are constants $\delta>0$, $K>0$, $\nu>0$ and $\alpha \in (0,1)$ such that the following hold.
\begin{enumerate}
\item The coefficients $c$ and $b_d$ obey
\begin{align}
\label{eq:ZerothOrderTermUpperBound}
c(t,x) &\leq K, \quad \forall\, (t,x)\in \overline{\HH}_{\infty},
\\
\label{eq:CoeffBD}
b_d(t,x',0) &\geq \nu, \quad \forall\, (t,x')\in [0,\infty)\times\RR^{d-1}.
\end{align}
\item On $\overline{\HH}_{2,\infty}$ (that is, near $x_d=0$), we require that
\begin{gather}
\label{eq:NonDegeneracyNearBoundary}
\sum_{i,j=1}^d a_{ij} (t,x) \eta_i \eta_j \geq \delta |\eta|^2, \quad \forall\, \eta \in \mathbb{R}^d, \quad\forall\, (t,x)\in \overline{\HH}_{2,\infty},
\\
\label{eq:BoundednessNearBoundary}
\max_{1\leq i,j\leq d}\|a_{ij}\|_{C(\overline{\HH}_{2,\infty})} + \max_{1\leq i\leq d}\|b_i\|_{C(\overline{\HH}_{2,\infty})} + \|c\|_{C(\overline{\HH}_{2,\infty})} \leq K,
\end{gather}
and, for all $P_1, P_2 \in \overline{\HH}_{2,\infty}$ such that $P_1 \neq P_2$ and $s(P_1, P_2) \leq 1$,
\begin{equation}
\label{eq:LocalHolderS}
\begin{aligned}
\max_{1\leq i,j\leq d}\frac{|a_{ij}(P_1)-a_{ij}(P_2)|}{s^{\alpha}(P_1,P_2)} &\leq K,
\\
\max_{1\leq i\leq d}\frac{|b_i(P_1)-b_i(P_2)|}{s^{\alpha}(P_1,P_2)} &\leq K,
\\
\frac{|c(P_1)-c(P_2)|}{s^{\alpha}(P_1,P_2)} &\leq K.
\end{aligned}
\end{equation}
\item On $\overline{\HH}_{\infty} \less \HH_{2,\infty}$ (that is, farther away from $x_d=0$), we require that
\begin{align}
\label{eq:NonDegeneracyInterior}
\sum_{i,j=1}^d x_da_{ij} (t,x) \eta_i \eta_j \geq \delta |\eta|^2, \quad \forall\, \eta \in \mathbb{R}^d, \quad\forall\, (t,x)\in \overline{\HH}_{\infty} \less \HH_{2,\infty},
\end{align}
and, for all $P_1, P_2 \in \overline{\HH}_{\infty} \less \HH_{2,\infty}$ such that $P_1 \neq P_2$ and $\rho(P_1, P_2) \leq 1$,
\begin{equation}
\label{eq:LocalHolderRho}
\begin{aligned}
\max_{1\leq i,j\leq d}\frac{|x_d^1a_{ij}(P_1)-x_d^2a_{ij}(P_2)|}{\rho^{\alpha}(P_1,P_2)} &\leq K,
\\
\max_{1\leq i\leq d}\frac{|b_i(P_1)-b_i(P_2)|}{\rho^{\alpha}(P_1,P_2)} &\leq K,
\\
\frac{|c(P_1)-c(P_2)|}{\rho^{\alpha}(P_1,P_2)} &\leq K.
\end{aligned}
\end{equation}
\end{enumerate}
\end{assump}

\begin{rmk}[Local H\"older conditions on the coefficients]
The local H\"older conditions \eqref{eq:LocalHolderS} and \eqref{eq:LocalHolderRho} are similar to those in \cite[Hypothesis 2.1]{Krylov_Priola_2010}.
\end{rmk}

\begin{rmk}[Linear growth of the coefficients of the parabolic differential operator]
\label{rmk:LinearGrowth}
Conditions \eqref{eq:BoundednessNearBoundary} and \eqref{eq:LocalHolderRho} imply that the coefficients $x_d a_{ij}(t,x)$, $b_i(t,x)$ and $c(t,x)$ can have at most linear growth in $x$. In particular, we have
\begin{equation}
\label{eq:LinearGrowth}
\sum_{i,j=1}^d|x_da_{ij}(t,x)| + \sum_{i=1}^d|b_i(t,x)| + |c(t,x)| \leq K(1+|x|), \quad \forall\, (t,x) \in \overline{\HH}_{\infty},
\end{equation}
for a sufficiently large positive constant, $K$.
\end{rmk}

\subsection{Existence and uniqueness of solutions to a degenerate-parabolic partial differential equation with unbounded coefficients}
\label{subsec:MainExistenceUniquenessPDE}
We now recall our main result from \cite{Feehan_Pop_mimickingdegen_pde}. The existence of solutions asserted by Theorem \ref{thm:MainExistenceUniquenessPDE} is used in Proposition \ref{prop:MarginalsUniqueness} to prove uniqueness of the one-dimensional marginal distributions of solutions to the mimicking stochastic differential equation \eqref{eq:MimickingSDE}, and is used in Theorem \ref{thm:MainMarginalsMatching} to prove that the one-dimensional marginal distributions of the It\^o process \eqref{eq:ItoProcess} match those of the solution to the mimicking stochastic differential equation \eqref{eq:MimickingSDE} satisfying the same initial condition.

\begin{thm} [Existence and uniqueness of solutions to a degenerate-parabolic partial differential equation with unbounded coefficients]
\label{thm:MainExistenceUniquenessPDE}
\cite[Theorem 1.1]{Feehan_Pop_mimickingdegen_pde}
Assume that the coefficients $(a,b,c)$ in \eqref{eq:Generator} obey the conditions in Assumption \ref{assump:Coeff}. Then there is a positive constant $p$, depending only on the H\"older exponent $\alpha\in(0,1)$, such that for any $T>0$, $f \in \sC^{\alpha}_p(\overline{\HH}_T)$ and $g \in \sC^{2+\alpha}_p(\overline{\HH})$, there exists a unique solution
$u \in \sC^{2+\alpha}(\overline{\HH}_T)$ to \eqref{eq:Problem}. Moreover, $u$ satisfies the a priori estimate
\begin{equation}
\label{eq:GlobalEstimate}
\|u\|_{\sC^{2+\alpha}(\overline{\HH}_T)}
\leq C \left(  \|f\|_{\sC^{\alpha}_p(\overline{\HH}_T)} +  \|g\|_{\sC^{2+\alpha}_p(\overline{\HH})} \right),
\end{equation}
where $C$ is a positive constant, depending only on $d$, $\alpha$, $T$, and the constants $K$, $\nu$, $\delta$
in Assumption \ref{assump:Coeff}.
\end{thm}

We note that in Theorem \ref{thm:MainExistenceUniquenessPDE} it is not necessary to specify a boundary condition along the portion $(0,T)\times\partial\HH$ of parabolic boundary of the domain $(0,T)\times\HH$ in order to obtain uniqueness. See the introductions to \cite{Feehan_parabolicmaximumprinciple, Feehan_maximumprinciple, Feehan_perturbationlocalmaxima} for more detailed discussions and examples illustrating this feature of `boundary-degenerate' parabolic operators.

\section{Martingale problem and the mimicking theorem}
\label{sec:MimickingTheorem}
In this section, we prove Theorem \ref{thm:MainWeakExistenceUniquenessSDE} concerning the degenerate stochastic differential equation with unbounded coefficients \eqref{eq:MimickingSDE}, and establish the main result, Theorem \ref{thm:MainMarginalsMatching}. Our method of proof combines ideas from the martingale problem formulation of Stroock and Varadhan \cite{Stroock_Varadhan} and the existence of solutions in suitable H\"older spaces, $\sC^{2+\alpha}(\overline{\HH}_T)$, to the homogeneous version of the initial value problem established in Theorem \ref{thm:MainExistenceUniquenessPDE}. In \S \ref{subsec:ExistenceSDE}, we prove existence of weak solutions to the mimicking stochastic differential equation \eqref{eq:MimickingSDE} and the existence of solutions to the martingale problem associated to $\sA_t$. In \S \ref{subsec:UniquenessSDE}, we establish uniqueness in law of solutions to \eqref{eq:MimickingSDE} and to the martingale problem for $\sA_t$, thus proving Theorems \ref{thm:MainWeakExistenceUniquenessSDE} and \ref{thm:MainExistenceUniquenessMartProb}; in \S \ref{subsec:MatchingOneDimMarginals}, we establish the matching property for the one-dimensional probability distributions for solutions to \eqref{eq:MimickingSDE} and of an It\^o process, thus proving Theorem \ref{thm:MainMarginalsMatching}.

\subsection{Existence of solutions to the martingale problem and of weak solutions to the stochastic differential equation}
\label{subsec:ExistenceSDE}
In this subsection, we show that \eqref{eq:MimickingSDE} has a weak solution
$(\widehat X,\widehat W)$, $\left(\Omega, \sF, \PP\right)$, $\{\sF_t\}_{t\geq 0}$
\cite[Definition 5.3.1]{KaratzasShreve1991}, for any initial point $x\in\overline{\HH}$, by proving existence of solutions to the martingale problem associated to $\sA_t$ (Definition \ref{defn:Martingale_Problem}).

We begin with an intuitive property of solutions to \eqref{eq:MimickingSDE} defined by an initial condition in $\overline\HH$. For this purpose, we consider coefficients defined on $[0,\infty)\times\RR^d$, instead of $[0,\infty)\times\overline\HH$. While Proposition \ref{prop:LawSupport} could also be proved using the generalized It\^o's formula \cite[Theorem 3.7.1 (v)]{KaratzasShreve1991} applied to the convex function $\varphi(x)=x_d^+$, for all $x\in\RR^d$, rather than the smooth function $\varphi$ in \eqref{eq:LawSupport1} in our proof below, we would also need to study the properties of the local time of $\widehat X_d(t)$ at zero. The fact that the local time of $\widehat X_d(t)$ at zero is trivial is proved in \cite[Lemma 6.6]{Bayraktar_Kardaras_Xing_2012} for a simpler process. The simple proof we give below avoids the need to consider the local time of the process $\widehat X_d(t)$.

\begin{prop}[Solutions started in a half-space remain in a half-space]
\label{prop:LawSupport}
Let
\begin{align*}
\tilde\sigma &:[0,\infty)\times\RR^d \rightarrow \RR^{d\times d},
\\
\tilde b &:[0,\infty)\times\RR^d \rightarrow \RR^d,
\end{align*}
be Borel measurable functions. Assume that
\begin{align}
\label{eq:LawSupportDiffusionZero}
\tilde \sigma(t,x)=0 \quad \hbox{when   } x_d <0,
\end{align}
and $\tilde b$ satisfies
\begin{equation}
\label{eq:FormOfb_d}
\tilde b_d(t,x) \geq 0 \quad \hbox{ when } x_d<0.
\end{equation}
If $(\widehat X,\widehat W)$, $\left(\Omega, \sF, \PP\right)$, $\{\sF_t\}_{t\geq 0}$ is a weak solution to
\[
d \widehat X = \tilde b(t, \widehat X(t)) dt + \tilde \sigma(t, \widehat X(t)) d \widehat W(t),\quad t\geq s,
\]
such that $\widehat X(s)\in\overline{\HH}$, then
\begin{equation}
\label{eq:LawSupport}
\PP \left(\widehat X(t)\in\overline{\HH}\right)=1, \quad \forall\, t \geq s.
\end{equation}
\end{prop}

\begin{proof}
It is sufficient to show that for any $\eps>0$, we have
\begin{equation}
\label{eq:LawSupport2}
\mathbb{P} \left(\widehat{X}_d(t) \in (-\infty,-\eps)\right) = 0, \quad \forall\, t \geq s.
\end{equation}
Let $\varphi$ be a smooth cutoff function on $\RR$ such that
\begin{equation}
\label{eq:LawSupport1}
0 \leq \varphi\leq 1 \hbox{ and } \varphi' \leq 0 \quad\hbox{on }\RR,
\quad
\varphi(x) = 1 \quad\hbox{for } x < -\eps,
\quad
\varphi(x) = 0 \quad\hbox{for } x > 0,
\end{equation}
and $|\varphi'|$ is bounded on $\RR$. It\^o's formula \cite[Theorem 3.3.3]{KaratzasShreve1991} implies that
\begin{align*}
\varphi(\widehat{X}_d(t))
&= \varphi(\widehat{X}_d(s)) + \int_{s}^{t} \sum_{i=0}^d\tilde \sigma_{di} (v, \widehat{X}(v)) \varphi'(\widehat{X}_d(v)) \,d\widehat W_i(v)\\
&\quad+\int_{s}^{t} \left[\tilde b_d(v, \widehat{X}(v)) \varphi'(\widehat{X}_d(v))+\frac{1}{2}(\tilde\sigma\tilde\sigma^*)_{dd}(v, \widehat{X}(v)) \varphi''(\widehat{X}_d(v))\right]\,dv,
\end{align*}
and so, because $\supp\varphi\subset (-\infty,0]$ and $\tilde \sigma$ obeys \eqref{eq:LawSupportDiffusionZero},
we have
\begin{equation*}
\varphi(\widehat{X}_d(t))
= \varphi(\widehat{X}_d(s)) + \int_{s}^{t} \tilde b_d(v, \widehat{X}(v)) \varphi'(\widehat{X}_d(v))\,dv.
\end{equation*}
By \eqref{eq:FormOfb_d}, \eqref{eq:LawSupport1} and the fact that $|\varphi'|$ is bounded on $\RR$, the integral term in the preceding identity is well-defined and is  non-positive. Therefore, we must have $\varphi(\widehat{X}_d(t)) \leq 0$ and hence $\varphi(\widehat{X}_d(t)) = 0$, for any choice of $\eps>0$, from where \eqref{eq:LawSupport2} and then \eqref{eq:LawSupport} follow.
\end{proof}

\begin{rmk}[Weak solutions are independent of choice of extension of coefficients to lower half-space]
\label{rmk:LawSupport}
Let
\[
\tilde b^i :[0,\infty)\times\RR^d \rightarrow \RR^d, \quad i=1,2,
\]
be measurable functions which satisfy condition \eqref{eq:FormOfb_d}, and assume
\begin{equation}
\label{eq:CoeffAgree}
\tilde b^1 = \tilde b^2 \quad \hbox{   on   } \quad [0,\infty)\times\overline\HH.
\end{equation}
Let $\tilde \sigma$ be a measurable function as in the hypotheses of Proposition \ref{prop:LawSupport}. Let $\widehat X$ be a weak solution to
\begin{equation}
\begin{aligned}
\label{eq:ExtendedSDE1}
d \widehat X(t) &= \tilde b^1(t,\widehat X(t)) dt +\tilde \sigma(t, \widehat X(t)) d \widehat W(t), \quad \forall\, t \geq s,
\end{aligned}
\end{equation}
such that
\[
\PP\left(\widehat X(s) \in \overline\HH\right)=1.
\]
Then, Proposition \ref{prop:LawSupport} shows that $\widehat X(t)$ remains supported in $\overline\HH$, for all $t \geq s$. By \eqref{eq:CoeffAgree}, it follows that $\widehat X$ is a weak solution to
\begin{equation}
\begin{aligned}
\label{eq:ExtendedSDE2}
d \widehat X(t) &= \tilde b^2(t,\widehat X(t)) dt +\tilde \sigma(t, \widehat X(t)) d \widehat W(t), \quad \forall\, t \geq s.
\end{aligned}
\end{equation}
This simple observation shows that, under the hypotheses of Proposition \ref{prop:LawSupport}, any weak solution started in $\overline \HH$ to \eqref{eq:ExtendedSDE1} is a weak solution to \eqref{eq:ExtendedSDE2}, and vice versa.
\end{rmk}

We have the following consequence of the existence theorem \cite[Theorem 5.3.10]{Ethier_Kurtz} of weak solutions for stochastic differential equations with unbounded, continuous coefficients. The main difference between \cite[Theorem 5.3.10]{Ethier_Kurtz} and Theorem \ref{thm:WeakExistence_SDE} is that the coefficients, $b$ and $\sigma$, in \eqref{eq:MimickingSDE} are defined on $[0,\infty)\times\bar\HH$, while the coefficients of the stochastic differential equation considered in \cite[Theorem 5.3.10]{Ethier_Kurtz} are defined on $[0,\infty)\times\RR^d$.

\begin{thm}[Existence of weak solutions to a stochastic differential equation with continuous coefficients]
\label{thm:WeakExistence_SDE}
Assume that the coefficients $\sigma$ and $b$ in \eqref{eq:MimickingSDE} are continuous on $[0,\infty)\times\overline\HH$, that
\begin{align}
\label{eq:Boundary_cond_sigma}
\sigma(t,x)&=0, \quad\forall\, (t,x)\in [0,\infty)\times\partial\HH,
\\
\label{eq:Boundary_cond_b}
b_d(t,x)&\geq 0, \quad\forall\, (t,x)\in [0,\infty)\times\partial\HH,
\end{align}
and that $\sigma\sigma^*$ and $b$ have at most linear growth in the spatial variable, uniformly in time, that is, there is a positive constant, $C$, such that
\begin{equation}
\label{eq:Linear_growth_unif_time}
|\sigma(t,x)\sigma^*(t,x)|+|b(t,x)| \leq C(1+|x|),\quad\forall\, (t,x)\in [0,\infty)\times\overline\HH.
\end{equation}
Then, for any $(s,x)\in[0,\infty)\times\overline{\HH}$, there exist weak solutions $(\widehat X, \widehat W)$, $(\Omega, \sF, \PP)$, $(\sF_t)_{t\geq s}$, to \eqref{eq:MimickingSDE} such that $\widehat X(s)=x$.
\end{thm}

\begin{proof}
Because $\sigma$ is continuous on $[0,\infty)\times\overline\HH$ and satisfies condition \eqref{eq:Boundary_cond_sigma}, we may extend $\sigma$ as a continuous function to $[0,\infty)\times\RR^d$ such that \eqref{eq:LawSupportDiffusionZero} is satisfied. We denote this extension by $\tilde\sigma\in C_{\loc}([0,\infty)\times\RR^d)$. Similarly, using \eqref{eq:Boundary_cond_b}, we can choose an extension, $\tilde b\in C_{\loc}([0,\infty)\times\RR^d) $, of the coefficient $b$ in \eqref{eq:DefinitionMimickingCoeff_b}, such that \eqref{eq:FormOfb_d} is satisfied. By \cite[Theorem 5.3.10]{Ethier_Kurtz},
there is a weak solution, $(\widehat X, \widehat W)$, $(\Omega, \sF, \PP)$, $(\sF_t)_{t\geq s}$,
to the stochastic differential equation
\[
d \widehat X = \tilde b(t, \widehat X(t)) dt + \tilde \sigma(t, \widehat X(t)) d \widehat W(t),\quad t>s,\quad \widehat X(s)=x.
\]
Because we assume $\widehat X(s)=x \in \bar \HH$, Proposition \ref{prop:LawSupport} shows that $\PP(\widehat X(t)\in\bar\HH)=1$, for all $t \geq s$, and so, we see that $(\widehat X, \widehat W)$ is also a weak solution to \eqref{eq:MimickingSDE}.
\end{proof}

We obtain the following consequence of Theorem \ref{thm:WeakExistence_SDE}

\begin{thm}[Existence of weak solutions to a martingale problem with continuous coefficients]
\label{thm:WeakExistence_martingale}
Assume that the coefficients $a$ and $b$ in \eqref{eq:MartingaleGenerator} are such that
\begin{align*}
&a \in C_{\loc}([0,\infty)\times\overline\HH; \SSS^d),
\\
&b \in C_{\loc}([0,\infty)\times\overline\HH; \RR^d).
\end{align*}
In addition, we assume that $a(t,x)$ is a positive definite matrix for every $(t,x)\in[0,\infty)\times\overline{\HH}$, that $b_d$ satisfies condition \eqref{eq:Boundary_cond_b}, and that $(x_da,b)$  satisfy condition \eqref{eq:LinearGrowth}. Then, for any $(s,x)\in[0,\infty)\times\overline{\HH}$, there is a solution, $\widehat \PP^{s,x}$, to the martingale problem associated to $\sA_t$ defined by \eqref{eq:MartingaleGenerator} such that \eqref{eq:Initial_Condition_Martingale_Problem} holds.
\end{thm}

\begin{proof}
Because $a$ is a positive definite matrix and it is continuous, by \cite[Lemma 6.1.1]{FriedmanSDE}, we may choose $\varsigma\in C_{\loc}([0,\infty)\times\overline\HH;\RR^{d\times d})$ such that \eqref{eq:Definition_a} holds. Now, by defining
\begin{equation}
\label{eq:Definition_sigma}
\sigma(t,x) := \sqrt{x_d} \varsigma(t,x),\quad\forall\, (t,x)\in [0,\infty)\times\overline\HH,
\end{equation}
we see that \eqref{eq:FormOfDiffusionMatrix} and \eqref{eq:Boundary_cond_sigma} hold. Since
\begin{equation}
\label{eq:sigma_sigma*_equalto_xda}
\sigma(t,x)\sigma^*(t,x) = x_da(t,x), \quad\forall\, (t,x)\in [0,\infty)\times\overline\HH,
\end{equation}
we see that \eqref{eq:Linear_growth_unif_time} is also obeyed (because of \eqref{eq:LinearGrowth}) and hence that the hypotheses on $(\sigma,b)$ in the statement of Theorem \ref{thm:WeakExistence_SDE} are satisfied. Let $\widehat \PP^{s,x}$ be the probability measure on
$$
(C_{\loc}([0,\infty);\overline\HH), \sB(C_{\loc}([0,\infty);\overline\HH)))
$$
induced by the weak solution $\widehat X$ to \eqref{eq:MimickingSDE} with coefficients $(\sigma,b)$ and initial condition $(s,x)$, given by Theorem \ref{thm:WeakExistence_SDE}. Then, \cite[Problem 5.4.3]{KaratzasShreve1991} implies that $\widehat \PP^{s,x}$ is a solution to the martingale problem associated to $\sA_t$ and satisfies \eqref{eq:Initial_Condition_Martingale_Problem}.
\end{proof}

\subsection{Uniqueness of solutions to the martingale problem and of weak solutions to the stochastic differential equation}
\label{subsec:UniquenessSDE}
We show that uniqueness in the sense of probability law holds for the weak solutions of the stochastic differential equation \eqref{eq:MimickingSDE}, with initial condition $x \in \overline\HH$, and we establish the well-posedness of the martingale problem associated to \eqref{eq:MimickingSDE}. First, we prove that uniqueness of the one-dimensional marginal distributions holds for weak solutions to \eqref{eq:MimickingSDE}, and then the analogue of \cite[Proposition 5.4.27]{KaratzasShreve1991} is used to show that uniqueness in law of solutions also holds.

We begin with the following version of the standard It\^o's formula (compare \cite[Theorem 3.3.6]{KaratzasShreve1991}) which applies to It\^o processes which are solutions to \eqref{eq:MimickingSDE}. The standard It\^o's formula cannot be applied directly to a solution $v$ to the terminal value problem \eqref{eq:MarginalUniqueness3} because the function $v$ is not $C^2$ up to the boundary (we only know that $x_d D^2 v$ is continuous up to the boundary). So we shall apply the standard It\^o's formula to the
shifted processes $\widehat X^{\eps}$ defined as in \eqref{eq:Shifted_process}, let $\eps$ tend to zero, and use condition \eqref{eq:CondDerivatives} to prove the appropriate convergence. An alternative proof of Proposition \ref{prop:ItoLemma} which does not appeal to \cite[Theorem IV.2.32]{Protter} is given in the proof of \cite[Proposition 3.4]{Feehan_Pop_mimickingdegen_probability_v1}, in an earlier version of this article.

\begin{prop}[It\^o's formula for functions which are not $C^2$ up to the boundary]
\label{prop:ItoLemma}
Assume that the coefficients $\sigma$ and $b$ of \eqref{eq:MimickingSDE} are continuous on $[0,\infty)\times\overline\HH$, that $\sigma$ obeys condition \eqref{eq:FormOfDiffusionMatrix}, the coefficient matrix $a$ defined by \eqref{eq:Definition_a} is continuous on $[0,T)\times\overline\HH$, and $\sigma\sigma^*$ and $b$ satisfy \eqref{eq:Linear_growth_unif_time}. Let $v \in C_{\loc}([0,\infty)\times\overline\HH)$ satisfy
\begin{align}
\label{eq:CondDerivatives}
& v_t, v_{x_i}, x_d v_{x_ix_j} \in C_{\loc}([0,\infty)\times\overline\HH), \quad\hbox{for } 1\leq i,j \leq d.
\end{align}
Let $(\widehat X, \widehat W)$ be a weak solution to \eqref{eq:MimickingSDE} on a filtered probability space $(\Omega,\PP,\sF)$, $\{\sF_t\}_{t\geq 0}$, such that $\widehat X(0) \in \overline\HH$, $\PP$-a.s. Then, the following holds $\PP$-a.s., for all
$s\leq t\leq T$,
\begin{equation}
\label{eq:ItoLemma}
\begin{aligned}
v(t,\widehat X(t))
&= v(s,\widehat X(s)) + \int_s^t \sum_{i,j=1}^d \sigma_{ij}(u,\widehat X(u)) v_{x_j} (u,\widehat X(u)) \,d \widehat W_j(u)\\
&\quad + \int_s^t\left(v_t(u,\widehat X(u))+\sum_{i=1}^d  b_i(u,\widehat X(u)) v_{x_i} (u,\widehat X(u)) \right.\\
&\quad\left.+\sum_{i,j=1}^d \frac{1}{2}\widehat X_d(u) a_{ij}(u,\widehat X(u)) v_{x_ix_j} (u,\widehat X(u)) \right) \,d u.
\end{aligned}
\end{equation}
\end{prop}

\begin{proof}
Without loss of generality, we may assume $s=0$ and it is sufficient prove \eqref{eq:ItoLemma} for $s=0$. We choose $\eps \geq 0$ and let
\begin{equation}
\begin{aligned}
x^{\eps} &:= (x_1,\ldots, x_{d-1}, x_d+\eps),
\\
\label{eq:Shifted_process}
\widehat X^{\eps}(u) &:= \left(\widehat X_{1}(u),\ldots, \widehat X_{d-1}(u), \widehat X_{d}(u)+\eps\right), \quad \forall\, u \geq 0.
\end{aligned}
\end{equation}
The proof follows by applying the standard It\^o's formula, \cite[Theorem 3.3.6]{KaratzasShreve1991}, to the processes
$\widehat X^{\eps}$ defined as in \eqref{eq:Shifted_process}, for $\eps>0$, and taking limit as $\eps$ tends to zero. This will require the use of condition \eqref{eq:CondDerivatives}. Consider the stopping times
\begin{equation*}
\tau_n := \inf\left\{u \geq 0: |\widehat X(u)|\geq n \right\} \quad \forall\, n \geq 1.
\end{equation*}
Since the coefficients $\sigma\sigma^*$ and $b$ satisfy \eqref{eq:Linear_growth_unif_time}, we obtain by \cite[Problem 5.3.15]{KaratzasShreve1991}, that for all $m \geq 1$ and $t \geq 0$, there is a positive constant $C=C(m,t,K,d)$ such that
\begin{align}
\label{eq:ItoLemma6}
\EE \left[ \max_{0\leq u\leq t} |\widehat X(u)|^{2m}\right] &\leq C\left(1+|x|^{2m}\right).
\end{align}
Then, it follows by \eqref{eq:ItoLemma6} that the non-decreasing sequence of stopping times $\{\tau_n\}_{n \geq 1}$ satisfies
\begin{equation}
\label{eq:StoppingTimesSequence}
\lim_{n \rightarrow \infty} \tau_n = +\infty \quad \PP\text{-a.s.}
\end{equation}
If this were not the case, then there would be a deterministic time $t > 0$ such that
\begin{equation}
\label{eq:ContradictionStoppingTimes}
\lim_{n \rightarrow\infty}\PP\left(\tau_n \leq t\right) >0.
\end{equation}
But, $\PP\left(\tau_n \leq t\right) = \PP\left(\sup_{0 \leq u \leq t}|\widehat X(u)| \geq n\right)$ and we have
\begin{align*}
\PP\left(\sup_{0 \leq u \leq t}|\widehat X(u)| \geq n\right)
&\leq \frac{1}{n^2} \EE \left[ \max_{0\leq u\leq t} |\widehat X(u)|^{2}\right]\\
&\leq \frac{C(1+|x|^2)}{n^2}, \quad \hbox{ (by \eqref{eq:ItoLemma6})}.
\end{align*}
Since the preceding expression converges to zero, as $n$ goes to $\infty$, we obtain a contradiction in \eqref{eq:ContradictionStoppingTimes}, and so \eqref{eq:StoppingTimesSequence} holds. By \eqref{eq:StoppingTimesSequence}, it suffices to prove \eqref{eq:ItoLemma} for the stopped process, that is
\begin{equation}
\label{eq:ItoLemmaStopped}
\begin{aligned}
v(t\wedge \tau_n,\widehat X(t\wedge \tau_n))
&= v(0,\widehat X(0)) + \int_0^{t\wedge \tau_n} \sum_{i,j=1}^d \sigma_{ij}(u,\widehat X(u)) v_{x_j} (u,\widehat X(u)) \,d \widehat W_j(u)\\
&\quad +
\int_0^{t\wedge \tau_n}\left(v_t(u,\widehat X(u))+\sum_{i=1}^d  b_i(u,\widehat X(u)) v_{x_i} (u,\widehat X(u)) \right.\\
&\quad\left.+\sum_{i,j=1}^d \frac{1}{2}\widehat X_d(u) a_{ij}(u,\widehat X(u)) v_{x_ix_j} (u,\widehat X(u)) \right) \,d u.
\end{aligned}
\end{equation}
By Proposition \ref{prop:LawSupport}, we have
\begin{equation}
\label{eq:ItoLemma4}
\widehat X(u) \in \overline \HH \quad \PP \text{-a.s.}\quad \forall\, u \in [0,T].
\end{equation}
Since $v\in C^{2}_{\loc}([0,T]\times\RR^{d-1}\times[\eps/2,\infty))$, we may extend $v$ to be a $C^{2}_{\loc}$ function on $[0,T]\times\RR^{d}$. Then we can apply the standard It\^o's formula, \cite[Theorem 3.3.6]{KaratzasShreve1991} and, taking into account the fact that $\widehat X_d(t) + \eps \geq \eps$, $\PP$-a.s., for all $t \geq 0$, we obtain
\begin{equation}
\label{eq:ItoLemma3}
\begin{aligned}
v(t\wedge \tau_n,\widehat X^{\eps}(t\wedge \tau_n)) &= v(0, \widehat X^{\eps}(0))+\int_0^{t\wedge \tau_n} \sum_{i,j=1}^d \sigma_{ij}(u,\widehat X(u)) v_{x_j}(u,\widehat X^{\eps}(u))\,d \widehat W_j(u)
\\
&\quad + \int_0^{t\wedge \tau_n}\left(v_t(u,\widehat X^{\eps}(u)) + \sum_{i=1}^d b_i(u,\widehat X(u)) v_{x_i}(u,\widehat X^{\eps}(u))\right.\\
&\quad \left.+\sum_{i,j=1}^d \frac{1}{2}\widehat X_d(u) a_{ij}(u,\widehat X(u))v_{x_ix_j}(u,\widehat X^{\eps}(u)) \right)\, du.
\end{aligned}
\end{equation}
Our goal is to show that, by taking the limit as $\eps \downarrow 0$, the left-hand and the right-hand side in
\eqref{eq:ItoLemma3} converge in probability to the corresponding expressions in \eqref{eq:ItoLemmaStopped}.

Since $v \in C_{\loc}(\overline\HH_T)$, we have for all $0\leq u\leq T$,
\begin{equation}
\label{eq:ItoLemma5}
v(u\wedge \tau_n, \widehat X^{\eps}(u\wedge \tau_n)) \rightarrow v(u\wedge \tau_n, \widehat X(u\wedge \tau_n)) \quad \PP \text{-a.s. when } \eps \downarrow 0.
\end{equation}
The terms in \eqref{eq:ItoLemma3} containing the pure It\^o integrals can be evaluated in the following way. We define, for all $i,j=1\ldots,d$ and $\eps \geq 0$,
$$
H^{\eps}_{ij}(u) := \sigma_{ij}(u,\widehat X(u)) v_{x_j}(u,\widehat X^{\eps}(u)) \mathbf{1}_{\{|\widehat X(u)| \leq n\}},\quad\forall\, u\in [0,T].
$$
Because $\sigma$ and $v_{x_j}$ are continuous on $[0,\infty)\times\overline\HH$ by \eqref{eq:CondDerivatives}, we see that the sequence $\{H^{\eps}_{ij}\}_{\eps \geq 0}$ is uniformly bounded on $[0,T]$, and converges $\PP$-a.s. to $H^0_{ij}$, for all $u \in [0,T]$. Then \cite[Theorem IV.2.32]{Protter} implies that, as $\eps\downarrow 0$,
\begin{align*}
\int_0^t H^{\eps}_{ij}(u)\,d \widehat W_j(u)
& \longrightarrow
\int_0^t H^0_{ij}(u)\,d \widehat W_j(u),
\end{align*}
where the convergence takes place in probability, as $\eps$ tends to zero. Using the fact that
$$
\int_0^{t\wedge \tau_n}  \sigma_{ij}(u,\widehat X(u)) v_{x_j}(u,\widehat X^{\eps}(u))\,d \widehat W_j(u)
=\int_0^t H^{\eps}_{ij}(u)\ d\widehat W_j(u),\quad\forall\, \eps \geq 0,
$$
we see that
\begin{equation}
\label{eq:Convergence_Ito_term}
\begin{aligned}
\int_0^{t\wedge \tau_n} \sigma_{ij}(u,\widehat X(u)) v_{x_j}(u,\widehat X^{\eps}(u))\,d \widehat W_j(u)
&\longrightarrow
\int_0^{t\wedge \tau_n} \sigma_{ij}(u,\widehat X(u)) v_{x_j}(u,\widehat X(u))\,d \widehat W_j(u),
\end{aligned}
\end{equation}
where the convergence takes place in probability as $\eps$ tends to zero.

For the $du$-term on the right-hand side of \eqref{eq:ItoLemmaStopped} we consider, for all $u\in [0,T]$ and all $\eps \geq 0$,
\begin{align*}
G^{\eps}(u) &:=
\left(v_t(u,\widehat X^{\eps}(u)) + \sum_{i=1}^d b_i(u,\widehat X(u)) v_{x_i}(u,\widehat X^{\eps}(u))\right.\\
&\qquad\left. +\sum_{i,j=1}^d \frac{1}{2}\widehat X_d(u) a_{ij}(u,\widehat X(u))v_{x_ix_j}(u,\widehat X^{\eps}(u))\right) \mathbf{1}_{\{|\widehat X(u)|\leq n\}}.
\end{align*}
Using again the fact that $a_{ij}$ and $x_d v_{x_ix_j}$ are continuous on $[0,\infty)\times\overline\HH$ and the fact that $\widehat X$ has continuous paths, by writing
\begin{align*}
\widehat X_d(u) a_{ij}(u,\widehat X(u))v_{x_ix_j}(u,\widehat X^{\eps}(u))
&= \frac{\widehat X_d(u)}{\widehat X^{\eps}_d(u)} a_{ij}(u,\widehat X(u))\widehat X^{\eps}_d(u)v_{x_ix_j}(u,\widehat X^{\eps}(u)),
\end{align*}
we see that, as $\eps\downarrow 0$,
\begin{align*}
\widehat X_d(u) a_{ij}(u,\widehat X(u))v_{x_ix_j}(u,\widehat X^{\eps}(u))
\longrightarrow
 a_{ij}(u,\widehat X(u))\widehat X_d(u)v_{x_ix_j}(u,\widehat X(u)),
\end{align*}
where the convergence takes place $\PP$-a.s., for all $u\in [0,T]$, as $\eps\downarrow 0$. Combining this with the fact that $b_i$, $v_t$, and $v_{x_i}$ are continuous on $[0,\infty)\times\overline\HH$ by \eqref{eq:CondDerivatives}, we see that the sequence $\{G^{\eps}\}_{\eps \geq 0}$ is uniformly bounded on $[0,T]$, and converges $\PP$-a.s. to $G^0$, for all $u \in [0,T]$. Then \cite[Theorem IV.2.32]{Protter} implies that
\begin{equation}
\label{eq:G_integral_convergence}
\int_0^t G^{\eps}(u)\ du \longrightarrow \int_0^t G^0(u)\ du
\quad\hbox{in probability as } \eps \downarrow 0.
\end{equation}
The convergence in \eqref{eq:G_integral_convergence} is equivalent to the fact that the $du$-term on the right-hand side of \eqref{eq:ItoLemmaStopped} converges in probability to the $du$-term on the right-hand side of \eqref{eq:ItoLemma}. By combining the latter convergence in probability with \eqref{eq:ItoLemma5} and \eqref{eq:Convergence_Ito_term}, we find that the right-hand side of \eqref{eq:ItoLemma3} converges in probability to the right-hand side in \eqref{eq:ItoLemmaStopped}, as $\eps$ tends to zero. This concludes the proof of the proposition.
\end{proof}

The next result is based on the \emph{existence} of a solution in $\sC^{2+\alpha}(\overline\HH_T)$ to the homogeneous initial value problem considered in Theorem \ref{thm:MainExistenceUniquenessPDE}. While it is not important in the proof of Proposition \ref{prop:MarginalsUniqueness} that the solution to problem \eqref{eq:Problem} is unique, it is important that a solution is
is smooth up to the boundary $\partial\HH$ in the sense of \eqref{eq:CondDerivatives}.

\begin{prop} [Uniqueness of the one-dimensional marginal distributions]
\label{prop:MarginalsUniqueness}
Assume the hypotheses of Theorem \ref{thm:MainWeakExistenceUniquenessSDE} hold. Let $(\widehat X^k, \widehat W^k)$, defined on filtered probability spaces $(\Omega^k,\PP^k,\sF^k)$, $\{\sF^k_t\}_{t\geq 0}$, $k=1,2$, be two weak solutions to \eqref{eq:MimickingSDE} with initial condition  $(s,x) \in [0,\infty)\times\overline{\HH}$. Then the one-dimensional marginal probability distributions of $\widehat X^1(t)$ and $\widehat X^2(t)$ agree for each $t\geq s$.
\end{prop}

\begin{proof}
We apply a duality argument as in the proof of \cite[Lemma 5.4.26]{KaratzasShreve1991} with the aid of Theorem \ref{thm:MainExistenceUniquenessPDE} and Proposition \ref{prop:ItoLemma}
but indicate the differences in the proof provided here. Without loss of generality, we may assume that $s=0$. By Proposition \ref{prop:LawSupport}, it is enough to show that for any $T>0$ and $g \in C^{\infty}_0(\overline{\HH})$, we have
\begin{equation}
\label{eq:MarginalUniqueness1}
\EE_{\PP^1}\left[g(\widehat X^1(T))\right]  = \EE_{\PP^2}\left[g(\widehat X^2(T))\right],
\end{equation}
where each expectation is taken under the law of the corresponding process. For this purpose, we consider the parabolic differential operator,
\begin{equation}
\label{eq:MatchingMarginals2_1}
-\mathring{L} w(t,x) := - w_t(t,x) + \sum_{i=1}^d  b_i(T-t,x) w_{x_i}(t,x)+\sum_{i,j=1}^d \frac{1}{2} x_d a_{ij}(T-t,x) w_{x_ix_j}(t,x),
\end{equation}
for all $(t,x) \in \HH_T$ and $w \in C^{2}(\HH_T)$. Let $u\in \sC^{2+\alpha}(\overline{\HH}_T)$ be the unique solution given by Theorem \ref{thm:MainExistenceUniquenessPDE} to the homogeneous initial value problem,
\begin{equation}
\label{eq:MatchingMarginals2_2}
\begin{cases}
\mathring{L} (t,x) = 0, &  \mbox{ for } (t,x) \in (0,T)\times\HH,\\
u(0,x) = g(x), & \mbox{ for } x \in   \overline{\HH}.
\end{cases}
\end{equation}
Define
\begin{equation}
\label{eq:MatchingMarginals2}
v(t,x):=u(T-t,x), \quad \forall\, (t,x)\in[0,T]\times\overline{\HH}.
\end{equation}
Then, $v\in \sC^{2+\alpha}(\overline{\HH}_T)$ solves the terminal value problem,
\begin{equation}
\label{eq:MarginalUniqueness3}
\begin{cases}
v_t(t,x)+\sA_t v(t,x)= 0, & \mbox{ for } (t,x)\in(0,T)\times\HH,\\
v(T,x) = g(x), & \mbox{ for } x\in\overline{\HH},
\end{cases}
\end{equation}
where the differential operator $\sA_t$ is given by \eqref{eq:MartingaleGenerator}.
Moreover, any function $v$ belonging to $\sC^{2+\alpha}(\overline{\HH}_T)$ obeys the boundary regularity property \eqref{eq:CondDerivatives} by the definition of $\sC^{2+\alpha}(\overline{\HH}_T)$ in \S \ref{subsec:DHKHolderSpaces}.
Therefore, Proposition \ref{prop:ItoLemma} gives us, for $k=1,2$,
\begin{equation}
\label{eq:MarginalUniqueness17}
\begin{aligned}
\EE_{\PP^k}\left[ v (T, \widehat X^k(T)) \right]
& = v(0,x) + \EE_{\PP^k}\left[\int_0^T \left(v_t+\sA_t \right)v (t, \widehat X^k(t)) \,dt\right]\\
&\quad+\EE_{\PP^k}\left[\int_0^T \sum_{i,j=1}^d \sigma_{ij}(t,\widehat X^k(t)) v_{x_j}(t, \widehat X^k(t)) \,d\widehat W^k_j(t)\right].
\end{aligned}
\end{equation}
Recall that $v_{x_i} \in C(\overline\HH_T)$ and that the coefficients $\sigma_{ij}$ satisfy \eqref{eq:LinearGrowth}. Inequality \eqref{eq:ItoLemma6} applied with $m=1$, gives
\[
\EE_{\PP^k}\left[\int_0^T \left|\sigma_{ij}(t,\widehat X^k(t)) v_{x_j}(t, \widehat X^k(t))\right|^2 \,d t\right] \leq C (1+|x|^2)\|v_{x_i}\|^2_{C(\overline\HH_T)},
\]
and so, the It\^o integrals in \eqref{eq:MarginalUniqueness17} are square-integrable, continuous martingales, which implies
\[
\EE_{\PP^k}\left[\int_0^T \sigma_{ij}(t,\widehat X^k(t)) v_{x_j}(t, \widehat X^k(t)) \,d\widehat W^k_j(t)\right] =0.
\]
Using the preceding inequality and \eqref{eq:MarginalUniqueness3}, we see that \eqref{eq:MarginalUniqueness17} yields
\begin{equation}
\label{eq:MatchingMarginals3}
\EE_{\PP^k}\left[g(\widehat X^k(T))\right] = v(0,x), \quad k=1,2,
\end{equation}
and so the equality \eqref{eq:MarginalUniqueness1} follows.
\end{proof}

Next, we recall

\begin{prop} [Uniqueness of solutions to the classical martingale problem]
\label{prop:KSLawUniqueness}
\cite[Proposition 5.4.27]{KaratzasShreve1991}
\cite[Theorem 4.4.2 \& Corollary 4.4.3]{Ethier_Kurtz}.
Let
\begin{align*}
& \tilde b:\RR^{d}\rightarrow \RR^d,\\
& \tilde \sigma:\RR^{d}\rightarrow \RR^{d \times d},
\end{align*}
be Borel measurable functions that are bounded on each compact subset in $\RR^d$. Define a differential operator by
\begin{align*}
\sG u(x) &:=  \sum_{i=1}^d \tilde b_i(x) u_{x_i} + \sum_{i,j=1}^d \frac{1}{2} \tilde a_{ij}(x) u_{x_ix_j}, \quad \forall\, x \in \RR^d,
\end{align*}
where $\tilde a:=\tilde\sigma\tilde\sigma^*$ and $u \in C^{2}(\RR^d)$. Suppose that for every $x \in \RR^d$, any two solutions $\PP^x$ and $\QQ^x$ to the time-homogeneous martingale problem associated with $\sG$ have the same one-dimensional marginal distributions. Then, for every initial condition $x \in \RR^d$, there exists at most one solution to the time-homogeneous martingale problem associated to $\sG$.
\end{prop}

We have the following consequence of Propositions \ref{prop:MarginalsUniqueness} and \ref{prop:KSLawUniqueness}; the difference between \cite[Proposition 5.4.27]{KaratzasShreve1991} and our Corollary \ref{cor:UniqueSolutionMartingaleProblem} is that the coefficients in \cite[Proposition 5.4.27]{KaratzasShreve1991} do not depend on time, while our coefficients do depend on time. The proof of Corollary \ref{cor:UniqueSolutionMartingaleProblem} proceeds by reducing the problem to one which Proposition \ref{prop:KSLawUniqueness} may be applied.

\begin{cor}[Uniqueness of solutions to the martingale problem associated to $\sA_t$]
\label{cor:UniqueSolutionMartingaleProblem}
Assume that the hypotheses of Theorem \ref{thm:WeakExistence_martingale} are satisfied. Suppose that for every $x \in \overline\HH$ and $s \geq 0$, any two solutions $\PP^{s,x}$ and $\QQ^{s,x}$ to the martingale problem in Definition \ref{defn:Martingale_Problem} associated to $\sA_t$ in \eqref{eq:MartingaleGenerator} with initial condition $(s,x)$ have the same one-dimensional marginal distributions. Then, for every initial condition $(s,x) \in [0,\infty)\times\overline\HH$, there exists at most one solution to the martingale problem associated to $\sA_t$.
\end{cor}

\begin{proof}
Just as in the proof of Theorem \ref{thm:WeakExistence_martingale}, we can find a matrix-valued function, $\sigma\in C_{\loc}([0,\infty)\times\overline\HH;\RR^{d\times d})$, obeying \eqref{eq:sigma_sigma*_equalto_xda}. As \eqref{eq:MimickingSDE} is time-inhomogeneous with initial condition $(s,x)\in[0,\infty)\times\overline{\HH}$, rather than time-homogeneous with initial condition $x\in\RR^d$, as assumed by Proposition \ref{prop:KSLawUniqueness}, we first extend the coefficients, $\sigma(t,x)$ and $b(t,x)$ as functions of $(t,x) \in [0,\infty)\times\overline\HH$, to coefficients $\tilde\sigma(t,x)$ and $\tilde b(t,x)$ as functions of $(t,x) \in \RR^{d+1}$ by setting, for $1\leq i,j \leq d$,
\begin{equation}
\label{eq:ExtentionMimickingCoeff1}
\begin{aligned}
\tilde\sigma_{ij}(t,x)
&:=
\begin{cases}
\sigma_{ij}(t,x) &\hbox{for } (t,x) \in [0,\infty)\times \overline\HH,
\\
0 &\hbox{otherwise, }
\end{cases}
\\
\tilde b_i(t,x)
&:=
\begin{cases}
b_i(t,x) &\hbox{for } (t,x) \in [0,\infty)\times \overline\HH,
\\
0 &\hbox{otherwise, }
\end{cases}
\end{aligned}
\end{equation}
The coefficients $\tilde\sigma_{ij}$ and $\tilde b_i$ defined in \eqref{eq:ExtentionMimickingCoeff1} are Borel measurable functions and need not be continuous. To obtain a time-homogeneous differential operator, as in Proposition \ref{prop:KSLawUniqueness}, we regard time, $t$, as an additional space coordinate and consider 
the following stochastic differential equation,
\begin{equation}
\begin{aligned}
\label{eq:AddTimeDimension}
dY_0(r)  &= dr,\quad \forall\, r \geq 0,
\\
d Y_i(r) &= \tilde b_i(Y(r)) dr + \sum_{j=1}^d \tilde \sigma_{ij}(Y(r)) dW_j(r), \quad i=1,\ldots,d,\quad \forall\, r \geq 0.
\end{aligned}
\end{equation}
Now, let $\sG$ denote the time-homogeneous differential operator,
\begin{align*}
\sG u(y) &:=  u_{y_0}+ \sum_{i=1}^d \tilde b_i(y) u_{y_i} + \sum_{i,j=1}^d \frac{1}{2} \tilde a_{ij}(y) u_{y_iy_j}, \quad \forall\, y \in \RR^{d+1},
\end{align*}
where $u \in C^2(\RR^{d+1})$.

For $x \in \overline\HH$ and $s \geq 0$, let $\PP^{s,x}$ and $\QQ^{s,x}$ be two solutions to the martingale problem associated to $\sA_t$ with initial condition $(s,x)$. We now extend the probability measures, $\PP^{s,x}$ and $\QQ^{s,x}$, from the measurable space $C_{\loc}([0,\infty);\overline\HH)$ to measures $\widetilde \PP^{s,x}$ and $\widetilde \QQ^{s,x}$
the canonical space $C_{\loc}([0,\infty);\RR^{d+1})$. We write functions $\omega$ in $C_{\loc}([0,\infty);\overline\HH)$ and $\tilde\omega$ in $C_{\loc}([0,\infty);\RR^{d+1})$, respectively, as
$$
\omega = \left(\omega_1, \ldots, \omega_d\right)
\quad\hbox{and}\quad
\tilde\omega = \left(\tilde\omega_0, \tilde\omega_1, \ldots, \tilde\omega_d\right).
$$
We define the probability measure $\widetilde \PP^{s,x}$ so that the distribution of the components $(\tilde\omega_1,\ldots, \tilde\omega_d)$ under the measure $\widetilde \PP^{s,x}$ is equal to the distribution of the components $(\omega_1,\ldots,\omega_d)$ under the measure $\PP^{s,x}$, that is
\begin{equation}
\begin{aligned}
\label{eq:Measures_agree}
&\widetilde \PP^{s,x} \left(\tilde\omega \in C_{\loc}([0,\infty);\RR^{d+1}): (\tilde\omega_1(r_i), \ldots, \tilde\omega_d(r_i)) \in B_i, i=1,\ldots,m\right) \\
&\quad
:=  \PP^{s,x} \left(\omega \in C_{\loc}([0,\infty);\overline\HH): (\omega_1(r_i), \ldots, \omega_d(r_i)) \in B_i, i=1,\ldots,m\right),
\end{aligned}
\end{equation}
for all $m \geq 1$, and $0\leq r_1 \leq r_2\leq\cdots\leq r_m$, and $B_i \in \sB(\RR^d)$, where $i=1,\ldots,m$. Let $\zeta_s(r) := r+s$, for all $r\geq 0$, be the shifted coordinate function. Keeping in mind \eqref{eq:AddTimeDimension}, we set
\begin{align}
\label{eq:Extension_measure}
\widetilde \PP^{s,x} \left\{\tilde\omega \in C_{\loc}([0,\infty);\RR^{d+1}): \tilde\omega_0 = \zeta_s\right\} :=1,
\end{align}
so that
\begin{align*}
\widetilde \PP^{s,x}\left\{\tilde\omega \in C_{\loc}([0,\infty);\RR^{d+1}):\tilde\omega_0\neq \zeta_s\right\} &= 1 - \widetilde \PP^{s,x}\left\{\tilde\omega \in C_{\loc}([0,\infty);\RR^{d+1}):\tilde\omega_0=\zeta_s\right\}
\\
&=0 \quad\hbox{(by \eqref{eq:Extension_measure})}.
\end{align*}
In the same way, we define an extension, $\widetilde\QQ^{s,x}$, of $\QQ^{s,x}$ from $C_{\loc}([0,\infty);\overline\HH)$ to $C_{\loc}([0,\infty);\RR^{d+1})$.

Notice that $\widetilde \PP^{s,x}$ and $\widetilde \QQ^{s,x}$ are both solutions to the classical time-homogeneous martingale problem (in the sense of \cite[Definition 5.4.15]{KaratzasShreve1991}) associated to $\sG$, with initial condition $(s,x)\in [0,\infty)\times\overline\HH \subset \RR^{d+1}$. Therefore, the probability measures $\PP^{s,x}$ and $\QQ^{s,x}$ will coincide if their extensions $\widetilde\PP^{s,x}$ and $\widetilde\QQ^{s,x}$ coincide. But, by Proposition \ref{prop:KSLawUniqueness}, the probability measures $\widetilde\PP^{s,x}$ and $\widetilde\QQ^{s,x}$ will coincide
if, for any $y=(y_0,\ldots,y_d)\in\RR^{d+1}$ and any two solutions $\widetilde\PP_i^{y}$, $i=1,2$, to the classical time-homogeneous martingale problem associated to $\sG$ with initial condition $y$, the one-dimensional marginal distributions of $\widetilde\PP_1^{y}$ and $\widetilde\PP_2^{y}$ coincide. We now appeal to \cite[Proposition 5.4.6 \& Corollary 5.4.8]{KaratzasShreve1991} to obtain a weak solution, $(Y^i,W^i)$, $(\Omega,\sF^i,\widetilde\PP^y_i)$, $\{\sF^i_r\}_{r\geq 0}$ for $i=1,2$, to \eqref{eq:AddTimeDimension} with initial condition $Y^i(0)=y$ such that the law of $Y^i$ is given by $\widetilde\PP^y_i$. Since the one-dimensional marginal distributions of the probability measures $\widetilde\PP_1^{y}$ and $\widetilde\PP_2^{y}$ agree if and only if the one-dimensional marginal distributions of the processes $Y^1$ and $Y^2$ agree,
it is enough to show that the latter holds. For this purpose, we consider two cases.

\setcounter{case}{0}
\begin{case}[$y_d < 0$ or $y_0<0$]
\label{case:y_d_or_y_0_negative}
In this case, the coefficients $\tilde b$ and $\tilde \sigma$ defined in \eqref{eq:ExtentionMimickingCoeff1} are identically zero on a neighborhood of $y$ in $\RR^{d+1}$, and so the unique solution, $Y$, to \eqref{eq:AddTimeDimension} with initial condition $Y(0)=y$
is given by
\begin{align}
\label{eq:y_d_or_y_0_negative}
Y(r)=y,\quad \forall\, r \geq 0.
\end{align}
Thus, $Y^1 = Y = Y^2$ and so the one-dimensional marginal distributions of the processes $Y^1$ and $Y^2$ trivially agree.
\end{case}

\begin{case}[$y_d \geq 0$ and $y_0 \geq 0$]
\label{case:y_d_or_y_0_nonnegative}
We claim that any weak solution, $(Y(r))_{r \geq 0}$, to \eqref{eq:AddTimeDimension} with initial condition $Y(0)=y$, has the property,
\begin{equation}
\label{eq:Y_dPositive}
Y_d(r) \geq 0 \quad \hbox{(almost surely)}, \quad\forall\, r \geq 0.
\end{equation}
Indeed, if this were not so, then there would be a constant $\eps>0$ such that $Y_d$ reached the level $-\eps$ with non-zero probability at some stopping time $\tau_{\eps}$. Because $Y_d$ has continuous paths, the process $Y_d$ would have hit $-\eps/2$ at a stopping time $\tau_{\eps/2} < \tau_{\eps}$, and from \eqref{eq:y_d_or_y_0_negative} in Case \ref{case:y_d_or_y_0_negative}, we would have
$$
Y_d(r) = -\eps/2, \quad\forall\, r \geq \tau_{\eps/2}.
$$
But this contradicts our assumption that $Y_d$ hits $-\eps$ at some time $\tau_{\eps} > \tau_{\eps/2}$, and therefore \eqref{eq:Y_dPositive} holds, as claimed.

Any weak solution, $(Y(r))_{r \geq 0}$, to \eqref{eq:AddTimeDimension} with initial condition $Y(0)=y$ induces a weak solution,
$(\widehat X(t))_{t \geq y_0}$,
\begin{equation}
\label{eq:FromYToX}
\widehat X(t) := (Y_1(t-y_0), Y_2(t-y_0), \ldots, Y_d(t-y_0)) \quad \forall\, t \geq y_0,
\end{equation}
to the stochastic differential equation
\begin{equation*}
\begin{aligned}
d \widehat X_i(t) &= \tilde b_i(t, \widehat X(t)) dt + \sum_{j=1}^d \tilde\sigma_{ij}(t, \widehat X(t)) dW_j(t), \quad i=1,\ldots,d,\quad \forall\, t \geq y_0,
\end{aligned}
\end{equation*}
with initial condition
\[
\widehat X(y_0)=(Y_1(0), \ldots, Y_d(0))=(y_1,\ldots,y_d)\in\overline\HH.
\]
Moreover, $\widehat X(t)$ remains in $\overline\HH$, for all $t \geq y_0$, by \eqref{eq:Y_dPositive}. The probability law of the process $(\widehat X(t))_{t \geq 0}$ coincides with the probability law of the process $(Y_1(r),\ldots, Y_d(r))_{r \geq 0}$, by \eqref{eq:Measures_agree} (with $y=(s,x)$ and $s=y_0$) since, by definition, $(Y(r))_{r \geq 0}$ is an extension of $(\widehat X(t))_{t \geq 0}$. By the hypothesis of Corollary \ref{cor:UniqueSolutionMartingaleProblem}, for every $x \in \overline\HH$ and $s \geq 0$, any two solutions $\PP^{s,x}$ and $\QQ^{s,x}$ to the martingale problem in Definition \ref{defn:Martingale_Problem} associated to $\sA_t$ in \eqref{eq:MartingaleGenerator} with initial condition $(s,x)$ have the same one-dimensional marginal distributions. Therefore, the one-dimensional marginal distributions of the process $(\widehat X(t))_{t \geq 0}$ are uniquely determined, which implies that the one-dimensional marginal distributions of the process $(Y_1(r),\ldots, Y_d(r))_{r \geq 0}$ are uniquely determined. Hence, the one-dimensional marginal distributions of the probability measures $\widetilde \PP^y_1$ and $\widetilde \PP^y_2$ agree for all $y\in\RR^{d+1}$ such that $y_0\geq 0$ and $y_d \geq 0$.
\end{case}

By combining the conclusions of Cases \ref{case:y_d_or_y_0_negative} and \ref{case:y_d_or_y_0_nonnegative}, we see that the one-dimensional marginal distributions of the probability measures $\widetilde \PP^y_1$ and $\widetilde \PP^y_2$ agree for all $y\in\RR^d$. Proposition \ref{prop:KSLawUniqueness} implies that $\widetilde \PP^y_1= \widetilde \PP^y_2$, for all $y\in \RR^{d+1}$. From \eqref{eq:Measures_agree}, we have $\PP^{s,x}=\QQ^{s,x}$, for all $(s,x)\in[0,\infty)\times\overline\HH$, which yields the desired uniqueness, and so the martingale problem associated to $\sA_t$ is well-posed.
\end{proof}

Finally, we have

\begin{proof}[Proof of Theorem \ref{thm:MainExistenceUniquenessMartProb}]
Theorem \ref{thm:WeakExistence_martingale} asserts the existence of solutions to the martingale problem associated to $\sA_t$.
We now describe how uniqueness of solutions to the martingale problem associated with $\sA_t$ follows from Proposition \ref{prop:MarginalsUniqueness} and Corollary \ref{cor:UniqueSolutionMartingaleProblem}. Let $\widehat \PP^{s,x}_i$, $i=1,2$, be two solutions to the martingale problem associated with $\sA_t$ with initial condition $(s,x)\in[0,\infty)\times\overline\HH$. Let $\sigma$ be defined as in \eqref{eq:Definition_sigma} in the proof of Theorem \ref{thm:WeakExistence_martingale}, where we showed that $\sigma$ satisfies \eqref{eq:FormOfDiffusionMatrix}, and let $b$ be as in \eqref{eq:MartingaleGenerator}. From Assumption \ref{assump:Coeff} (a hypothesis of Theorem \ref{thm:MainExistenceUniquenessMartProb})
we see that the coefficient functions $(\sigma, b)$ satisfy the conditions in Assumption \ref{assump:MimickingCoeffSpecialForm}, and so we may apply Proposition \ref{prop:MarginalsUniqueness} to conclude that any two weak solutions to \eqref{eq:MimickingSDE} with initial condition $(s,x)$  have the same one-dimensional marginal distributions for $t\geq s$. Let $\widehat X_i$, $i=1,2$, be the weak solutions to \eqref{eq:MimickingSDE} with initial condition $\widehat X_i(s)=x$ such that the law of $\widehat X_i$ is given by $\widehat\PP^{s,x}_i$ (see \cite[Proposition 5.4.6 \& Corollary 5.4.8]{KaratzasShreve1991}). Consequently, the one-dimensional marginal distributions agree for the probability measures, $\widehat\PP^{s,x}_i$, $i=1,2$, since they agree for the stochastic processes, $\widehat X_i$, $i=1,2$. We now apply Corollary \ref{cor:UniqueSolutionMartingaleProblem} to conclude that the probability measures, $\widehat\PP^{s,x}_i$, $i=1,2$, coincide and so we obtain the desired uniqueness of solutions to the martingale problem associated with $\sA_t$. Therefore, the martingale problem associated to $\sA_t$ is well-posed, for any initial condition $(s,x)\in[0,\infty)\times\overline\HH$.
\end{proof}

\begin{proof}[Proof of Theorem \ref{thm:MainWeakExistenceUniquenessSDE}]
By Theorem \ref{thm:WeakExistence_SDE}, we obtain existence of weak solutions to \eqref{eq:MimickingSDE}. Since each weak solution induces a probability measure on $C_{\loc}([0,\infty);\overline\HH)$ which solves the martingale problem associated to $\sA_t$, Theorem \ref{thm:MainExistenceUniquenessMartProb} implies that the probability law of the weak solutions to \eqref{eq:MimickingSDE} is uniquely determined.

The fact that the weak solutions to \eqref{eq:MimickingSDE} satisfy the strong Markov property can be shown to follow by the same argument applied in the time-homogeneous case in \cite[Theorem 4.4.2 (b) \& (c)]{Ethier_Kurtz}; an alternative argument is provided below.

To prove the strong Markov property of weak solutions to \eqref{eq:MimickingSDE}, we consider again the time-homogeneous stochastic differential equation \eqref{eq:AddTimeDimension} arising in the proof of Corollary \ref{cor:UniqueSolutionMartingaleProblem}. The same argument as the one used in the proof of Corollary \ref{cor:UniqueSolutionMartingaleProblem} to conclude that the martingale problem associated to $\sA_t$ is well-posed can be used to conclude that the classical martingale problem associated to the stochastic differential equation \eqref{eq:AddTimeDimension} is well-posed. Therefore, by \cite[Theorem 5.4.20]{KaratzasShreve1991}, we see that for any $y \in \RR^{d+1}$, the (necessarily unique) weak solution $Y=Y^y$ to \eqref{eq:AddTimeDimension} started at $y$ possesses the strong Markov property, that is, for any stopping time $\tau$ of $\{\sB_t(C_{\loc}([0,\infty);\RR^{d+1}))\}_{t \geq 0}$, Borel measurable set $B \in \sB(\RR^{d+1})$, and $u \geq 0$, we have
\begin{equation}
\label{eq:StrongMarkovY}
\widetilde \PP^y(Y(\tau+u) \in B | \sB_{\tau}(C_{\loc}([0,\infty);\RR^{d+1})) = \widetilde \PP^y(Y(\tau+u) \in B | Y(\tau)),
\end{equation}
where $\widetilde \PP^{y}$ denotes the probability law of the process $Y$ started at $y$. Let $(s,x) \in [0,\infty)\times\overline\HH$ and let $\widehat X=\widehat X^{s,x}$ be the unique weak solution to \eqref{eq:MimickingSDE} with initial condition $\widehat X^{s,x}(s)=x$. Observe that
\[
Y^{s,x} (r) = \left(r+s, \widehat X_1(r+s), \ldots, \widehat X_d(r+s)\right)\quad r \geq 0,
\]
is the (unique) solution to \eqref{eq:AddTimeDimension} with initial condition $Y^{s,x}(0)=(s,x)$. Therefore, \eqref{eq:StrongMarkovY} can be rewritten in terms of the probability law $\PP^{s,x}$ of $\widehat X^{s,x}$,
\begin{equation}
\label{eq:StrongMarkovX}
 \PP^{s,x}(\widehat X(\tau+u) \in B | \sB_\tau(C_{\loc}([0,\infty);\overline\HH)) =  \PP^{s,x}(\widehat X(\tau+u) \in B | X(\tau)),
\end{equation}
for any stopping time $\tau$ of $\{\sB_t(C_{\loc}([0,\infty);\overline\HH))\}_{t \geq 0}$, any Borel measurable set $B \in \sB(\overline\HH)$ and $u \geq s$. Thus, $\widehat X^{s,x}$ satisfies the strong Markov property.
\end{proof}

\subsection{Matching one-dimensional marginal probability distributions}
\label{subsec:MatchingOneDimMarginals}
We can now complete the proof of Theorem \ref{thm:MainMarginalsMatching}.

\begin{proof}[Proof of Theorem \ref{thm:MainMarginalsMatching}]
The existence of coefficient functions, $(\sigma,b)$, satisfying Assumption \ref{assump:MimickingCoeffSpecialForm} follows by the same argument employed at the beginning of the proof of Corollary \ref{cor:UniqueSolutionMartingaleProblem}.

Let $\widehat X$ be the unique weak solution to the mimicking stochastic differential equation \eqref{eq:MimickingSDE} with coefficients $(\sigma, b)$ and initial condition $\widehat X(0)= X(0)=x_0$, given by Theorem \ref{thm:MainWeakExistenceUniquenessSDE}. As in the proof of Proposition \ref{prop:MarginalsUniqueness}, using assumption \eqref{eq:SupportItoProcess}, we need to show that for any deterministic time $T\geq 0$ and function $g\in C^{\infty}_0(\overline{\HH})$, we have
\begin{equation}
\label{eq:MatchingMarginals1}
\EE\left[g(\widehat X(T))\right]  = \EE\left[g(X(T))\right].
\end{equation}
Notice that it is enough to consider functions, $g$, with support in $\overline\HH$ because both the It\^o process, $X$, and the solution to \eqref{eq:MimickingSDE}, $\widehat X$, are supported in $\overline\HH$ by \eqref{eq:SupportItoProcess} and Proposition \ref{prop:LawSupport}, respectively.

Let $v \in \sC^{2+\alpha}(\overline{\HH}_T)$ be defined as the unique solution to the terminal value problem \eqref{eq:MarginalUniqueness3}.
Then,  \eqref{eq:MatchingMarginals3} gives
\begin{equation}
\label{eq:MatchingMarginals4}
\EE\left[g(\widehat X(T))\right] = v(0,x_0).
\end{equation}
We wish to prove that \eqref{eq:MatchingMarginals4} holds with $X(T)$ in place of $\widehat X(T)$. We proceed as in the proof of Proposition \ref{prop:MarginalsUniqueness}. We apply the standard It\^o's formula to $v(t, X^{\eps}(t))$, where $X^{\eps}$ is defined in \eqref{eq:Shifted_process} with the role of $\widehat X$ now replaced by $X$. We cannot apply the version of It\^o's formula proved in Proposition \ref{prop:ItoLemma} to the process $X$ because it is an It\^o process and it does not necessarily have the special structure of $\widehat X$ (see Assumption \ref{assump:MimickingCoeffSpecialForm})
whose coefficients satisfy the hypotheses of Proposition \ref{prop:ItoLemma}, a fact which is directly used in the proof of Proposition \ref{prop:ItoLemma}. We obtain
\begin{align*}
dv(t, X^{\eps}(t))
&= \left(v_t(t, X^{\eps}(t)) + \sum_{i=1}^d \beta_i(t) v_{x_i}(t, X^{\eps}(t)) + \sum_{i,j=1}^d \frac{1}{2} (\xi\xi^*)_{ij}(t)v_{x_ix_j}(t, X^{\eps}(t)) \right) dt
\\
&\qquad + \sum_{i,j=1}^d \xi_{ij}(t) v_{x_i}(t, X^{\eps}(t)) d W_j(t).
\end{align*}
The $dW_j(t)$-terms in the preceding identity are square-integrable, continuous martingales, because
\[
[0,T] \ni t  \mapsto v_{x_i}(t, X^{\eps}(t))
\]
are bounded processes since $v_{x_i} \in C([0,T]\times\overline\HH)$, and $\xi(t)$ is an adapted process which is square-integrable by \eqref{eq:IntegrabilityCondition}. Therefore,
\begin{align*}
\EE\left[v(T, X^{\eps}(T))\right]
&=
v(0,x^{\eps}_0) +  \EE\left[\int_0^T \left(v_t(t,X^{\eps}(t)) + \sum_{i=1}^d \beta_i(t) v_{x_i}(t, X^{\eps}(t)) \right.\right.
\\
&\qquad + \left.\left. \frac{1}{2}\sum_{i,j=1}^d (\xi\xi^*)_{ij}(t)v_{x_ix_j}(t,X^{\eps}(t)) \right) \,dt\right].
\end{align*}
Using conditional expectations, we may rewrite the preceding identity as
\begin{align*}
\EE\left[v(T, X^{\eps}(T))\right]
&=
v(0,x^{\eps}_0) + \int_0^T\EE\left[\EE\left[\left(v_t(t,X^{\eps}(t)) + \sum_{i=1}^d\beta_i(t) v_{x_i}(t, X^{\eps}(t)) \right.\right.\right.
\\
&\qquad
+ \left.\left.\left.\left.\frac{1}{2}\sum_{i,j=1}^d (\xi\xi^*)_{ij}(t)v_{x_ix_j}(t,X^{\eps}(t))\right)
\right| X^{\eps}(t)\right]\right] \,dt
\\
&=
v(0,x^{\eps}_0) + \int_0^T\EE\left[\right.v_t(t,X^{\eps}(t)) + \sum_{i=1}^d\EE\left[\beta_i(t)\left.\right| X^{\eps}(t)\right] v_{x_i}(t, X^{\eps}(t))
\\
&\qquad
+\frac{1}{2}\sum_{i,j=1}^d \EE\left[(\xi\xi^*)_{ij}(t)\left.\right| X^{\eps}(t)\right] v_{x_ix_j}(t,X^{\eps}(t))\left.\right] \,dt,
\end{align*}
and thus, using \eqref{eq:DefinitionMimickingCoeff_b}, \eqref{eq:DefinitionMimickingCoeff_a} and \eqref{eq:MartingaleGenerator}, we see that
\begin{align}
\label{eq:Equality_v_eps}
\EE\left[v(T, X^{\eps}(T))\right] = v(0,x^{\eps}_0) + \EE\left[\int_0^T \left(v_t(t, X^{\eps}(t))+\sA_t v(t, X^{\eps}(t)) \right)\,dt\right].
\end{align}
But $v_t(t,x)+\sA_t v(t,x)=0$, for all $(t,x)\in \HH_T$, by \eqref{eq:MarginalUniqueness3}. Therefore, by letting $\eps \downarrow 0$ in the identity \eqref{eq:Equality_v_eps}, we obtain
\[
\EE\left[g( X(T))\right] = v(0,x_0),
\]
and by \eqref{eq:MatchingMarginals4} this concludes the proof.
\end{proof}

%
%

\bibliography{mfpde}
\bibliographystyle{amsplain}

\end{document}